\documentclass[12pt,a4paper]{amsart}

\usepackage[colorinlistoftodos]{todonotes}
\usepackage{breqn}
\usepackage{url} 

\usepackage{amsmath,amssymb,amsfonts,amsthm,amscd,textcomp,times,booktabs}
\allowdisplaybreaks
\usepackage{eclbkbox}
  \ifx\pdfoutput\undefined
   \fi
  \usepackage{braket}
  \usepackage{color,float}
   \usepackage[normalem]{ulem}
   \usepackage{tabularx}

\newtheorem{theorem}{Theorem}
\newtheorem{conjecture}[theorem]{Conjecture}

\newtheorem{proposition}[theorem]{Proposition}
\newtheorem{lemma}[theorem]{Lemma}
\newtheorem{definition}[theorem]{Definition}
\newtheorem{corollary}[theorem]{Corollary}
\newtheorem{remark}[theorem]{Remark}
\newtheorem{example}[theorem]{Example}
\newtheorem{problem}[theorem]{Problem}
\newtheorem{question}[theorem]{Question}

\newtheorem{step}{Step}
\numberwithin{equation}{section} \numberwithin{theorem}{section}

\usepackage[all]{xy}

\newcommand{\bs}{\boldsymbol}

\newcommand{\mapright}[1]{\smash{\mathop{   \hbox to 0.7cm{\rightarrowfill}}
 \limits^{#1}}}

\newcommand{\N}{\ensuremath{\mathbb{N}}}
\newcommand{\Z}{\ensuremath{\mathbb{Z}}}

\newcommand{\R}{\ensuremath{\mathbb{R}}}
\newcommand{\C}{\ensuremath{\mathbb{C}}}

\newcommand{\vol}{\ensuremath{\mathrm{vol}}}
\newcommand{\Vol}{\ensuremath{\mathrm{Vol}}}
\newcommand{\bracket}[1]{\ensuremath{\langle #1 \rangle}}

\DeclareMathOperator{\Hom}{Hom} 
 
\DeclareMathOperator{\NE}{NE}

\newcommand{\conv}{{\rm conv}}

\def\C{\mathbb C}

\def\P{\mathbb P}

\title{Toric Fano manifolds that do not admit \linebreak  extremal K\"ahler metrics}

\author{DongSeon Hwang}
\address{Center for Complex Geometry, Institute for Basic Science (IBS), Daejeon $34126$, Republic of Korea}
\email{dshwang@ibs.re.kr}

\author{Hiroshi Sato}
\address{Department of Applied Mathematics, Faculty of Sciences, Fukuoka University, $8$-$19$-$1$,
Nanakuma, Jonan-ku, Fukuoka $814$-$0180$, Japan}
\email{hirosato@fukuoka-u.ac.jp}

\author{Naoto Yotsutani}
\address{Kagawa University, Faculty of education, Mathematics, Saiwaicho $1$-$1$, Takamatsu, Kagawa, $760$-$8522$, Japan}
\email{yotsutani.naoto@kagawa-u.ac.jp}

\makeatletter
\@namedef{subjclassname@2020}{%
\textup{2020} Mathematics Subject Classification}
\makeatother

\subjclass[2020]{Primary: 53C55, Secondary: 14L24, 14M25}
\keywords{Fano variety, Extremal metrics, relative K-stability, toric variety}

\date{\today}

\begin{document}

\maketitle

\noindent{\bfseries Abstract.}
We show that there exists a toric Fano manifold of dimension $10$ that does not admit an extremal K\"ahler metric in the first Chern class, answering a question of Mabuchi. By taking a product with a suitable toric Fano manifold, one can also produce a toric Fano manifold of dimension $n$ admitting no extremal K\"ahler metric in the first Chern class for each $n \geq 11$.

\section{Introduction}
One of the central problems in K\"ahler geometry is to find a canonical K\"ahler metric on a given compact K\"ahler manifold. The most well-known and important metric of this kind is a K\"ahler--Einstein metric. The famous Calabi problem asks the existence of K\"ahler--Einstein metrics of a given compact K\"ahler manifold. It turned out that Calabi--Yau manifolds and canonically polarized manifolds always admit K\"ahler--Einstein metrics  by Aubin and Yau (\cite{Au} and \cite{Yau}). However, not every Fano manifold admits a K\"ahler--Einstein metric. The first obstruction was given by Matsushima (\cite{Mat}).   Thanks to the solution of the Yau--Tian--Donaldson conjecture for K\"ahler--Einstein metric case, the existence of the K\"ahler--Einstein metric can be described in terms of a stability condition in algebraic geometry (\cite{CDS15a},\cite{CDS15b}, \cite{CDS15c} and \cite{Ti15}). More precisely, an anti-canonically polarized Fano manifold $(X,-K_X)$ admits a K\"ahler--Einstein metric in the first Chern class if and only if $(X,-K_X)$ is K-polystable. However, verifying K-polystability for individual Fano manifolds remains a challenging problem and continues to be an active area of research.

Now it is natural to seek a suitable metric for a Fano manifold that does not admit a K\"ahler--Einstein metric. In the literature,    K\"ahler--Ricci solitons, Mabuchi solitons, and extremal K\"ahler metrics  are among the most well-known candidates for serving as canonical K\"ahler metrics.  Let $g$ be a K\"ahler metric on a Fano manifold $X$ of dimension $n$. Denote by $\omega_g$ the K\"ahler form of the K\"ahler metric $g$ representing $c_1(X)$. By the $\partial \bar{\partial}$-lemma, there is a unique $F_g \in C^\infty(X, \mathbb{R})$, called {\em the Ricci potential of $g$}, satisfying 
$$\mathrm{Ric}(\omega_g) - \omega_g = \frac{\sqrt{-1}}{2\pi}\partial\bar{\partial}F_g 
\qquad \text{and} \qquad \int_X (1-e^{F_g})\omega^n_g = 0.$$
Then the K\"ahler metric $g$ is called a \emph{K\"ahler--Ricci soliton} if the gradient vector field $\mathrm{grad}^{\mathbb{C}}_{\omega_g} F_g$ is a holomorphic vector field on $X$. It is called a \emph{Mabuchi soliton} or a \emph{generalized K\"ahler--Einstein metric} if $\mathrm{grad}^{\mathbb{C}}_{\omega_g} (1-e^{F_g})$ is a holomorphic vector field on $X$. It is called an \emph{extremal K\"ahler metric} if $\mathrm{grad}^{\mathbb{C}}_{\omega_g} (s(\omega_g)-n)$ is a holomorphic vector field on $X$ where $s(\omega_g)$ denotes the scalar curvature of $g$. The last notion is defined by Calabi with an equivalent condition in \cite{Ca82}.

A lot of work has been devoted to the YTD type conjectures for those canonical K\"ahler metrics. A  Fano variety admits a K\"ahler--Ricci soliton if and only if it is reduced uniformly Ding stable (\cite[Theorem 1.3]{BLXZ}). Combining the results in \cite{HL}, \cite{Yao22b}, and \cite{Hi}, an anti-canonically polarized Fano manifold admits a Mabuchi soliton if and only if it is uniformly relatively  Ding stable. 
The toric case was treated earlier in \cite{Yao22a}. 
On the other hand, the YTD conjecture for the extremal K\"ahler metric case is still open. 
\begin{conjecture}\cite[Conjecture 1.1]{Sz}
    A polarized manifold admits an extremal K\"ahler metric in the class of the polarization if and only if it is K-stable relative to a maximal torus of automorphisms, \emph{relatively K-polystable} in short.
\end{conjecture}

One implication turned out to hold true, while the other implication remains open. 
\begin{theorem}\cite[Theorem 1.4]{SS}\label{eK=>relK}
If a polarized manifold admits an extremal K\"ahler metric then it is relatively K-polystable.
\end{theorem}

Similar to the case of K\"ahler--Einstein metrics, verifying each stability condition is also a difficult problem for those canonical K\"ahler metrics.  But there have been some progress for K\"ahler--Ricci solitons and Mabuchi solitons when $X$ is a toric manifold. Every toric Fano manifold admits a K\"ahler--Ricci soliton by \cite[Theorem 1.1]{WZ04}. See also  \cite{SZ12} for the orbifold case.  The existence problem of Mabuchi soliton in low dimensional cases was considered in \cite{NSY23} based on the criterion given in  \cite{Yao22a}. In particular, the toric Fano manifold $X=\P_{\P^2}(\mathcal O\oplus \mathcal O(2))$ of dimension $3$ does not admit a Mabuchi soliton in the first Chern class, whereas it admits an extremal K\"ahler metric. 

Moreover, the above three canonical K\"ahler metrics exist on Fano surfaces (\cite{Ca82}, \cite{CLW} and \cite{Yao22a}) but not in higher dimensions for K\"ahler--Ricci solitons and Mabuchi solitions (\cite{MT} and \cite{NSY23}).  
There exist compact K\"ahler manifolds, that are not Fano, that do not admit an extremal K\"ahler metric (\cite{L}). However, to the best of the authors' knowledge, there is   no known example of a Fano manifold that does not admit an extremal K\"ahler metric 
in the literature. 

The folklore conjecture states that every {\em toric} Fano manifold admits an extremal Kahler metric in its first Chern class, motivated by the fact that every toric Fano manifold admits a K\"ahler--Ricci soliton. In particular, Mabuchi posed the following problem as an approach to the folklore conjecture. 

\begin{problem}\cite[Question 2]{Mab10}\label{MabuchiProblem}
Let $(X,-K_X)$ be a smooth polarized toric Fano manifold.
Is $(X,-K_X)$ always relatively K-polystable?
\end{problem}

However, we find a relatively K-unstable toric Fano manifold, which gives an explicit counterexample to the folklore conjecture by Theorem \ref{eK=>relK}. An instability criterion for the relative K-stability was proposed in \cite{YZ19}. But since this criterion is not optimal, we do not completely answer Problem \ref{MabuchiProblem} even in dimension $3$ as discussed in \cite{YZ23}. Despite this complication, we resolve Problem \ref{MabuchiProblem} in dimension $10$ based on the instability criterion in \cite{YZ19}. 

\begin{theorem}\label{thm:main}
    There exists a relatively K-unstable toric Fano manifold of dimension $10$.
\end{theorem}

Thanks to Theorem \ref{eK=>relK}, it implies the following. 

\begin{corollary}\label{noExtK}
 There exists a toric Fano manifold of dimension $10$ that does not admit an extremal K\"ahler metric in the first Chern class.
\end{corollary}

The toric Fano manifold $X_P$ of dimension $10$ is constructed based on observations regarding the $866$
toric Fano manifolds of dimension $5$. Even though no toric Fano manifold of dimension $5$ satisfies the instability criterion in (\ref{ineq:insta}), 
the $5$-dimensional toric Fano manifold 
with ID:788 according to \cite{Paffenholz} 
appears to be the most suitable candidate for our purpose. 
Thus, as a generalization of this $5$-dimensional toric Fano manifold, we construct a toric Fano manifold $\mathfrak{X}_r$ of dimension $5r$ for $r\in\N$ and prove that $\mathfrak{X}_2$ satisfies the instability criterion. 
It is expected that $\mathfrak{X}_r$ also has interesting geometric properties for $r\ge 3$. 
See Example \ref{10dim} for the explicit combinatorial description of the $10$-dimensional toric Fano manifold. 
In Example \ref{mainex}, we provide a precise geometric construction of $\mathfrak{X}_r$. 
Moreover, some geometric aspects of $\mathfrak{X}_r$ are discussed in Remark \ref{mainremark}. 

\begin{question}
    Is it always true that  $\mathfrak{X}_r$ does not admit an extremal K\"ahler metric for $r \geq 3$?
\end{question}
In order to provide some examples of toric Fano manifolds that do not admit extremal K\"ahler metrics in higher dimensions, we 
recall that the extremal K\"ahler metric on a product of polarized K\"ahler manifolds is always a product metric.

\begin{theorem}$($\cite[Theorem 1]{AH}, \cite[Theorem 1.2]{Hu}, \cite[Corollary 1.4]{ST}$)$ 
Let $(X, L)$ be the product of two polarized K\"ahler
manifolds $(X_1, L_1)$ and $(X_2, L_2)$. Assume that $X$ admits an extremal K\"ahler metric $g$ in the class $2\pi c_1(L)$.
Then $g$ is a product metric $g_1 \times g_2$, where each $g_i$ is an extremal metric on $X_i$ in the class $2\pi c_1(L_i)$. 
\end{theorem}
 
Thus, by taking a product of our example $\mathfrak{X}_2$ in Corollary \ref{noExtK} with {\it any} toric Fano manifold $X$, 
one can immediately produce an example $Y:=\mathfrak{X}_2\times X$ not admitting an extremal K\"ahler metric in every higher dimension. This has been pointed out by Nakamura \cite{Na24}. 

\begin{corollary}\label{noExtK2}
 For each $n \geq 10$, there exists a toric Fano manifold of dimension $n$ that does not admit an extremal K\"ahler metric
 in the first Chern class. 
\end{corollary}

It is observed that the existence of a Mabuchi soliton implies the existence of an extremal K\"ahler metric (see \cite[Theorem 9.8]{Ma21}, \cite[Remark 2.22]{Hi}, and \cite[Remark 5.8]{Yao22b}).  

\begin{theorem}
    If a Fano manifold admits a Mabuchi soliton, then it admits an extremal K\"ahler metric in the first Chern class.
\end{theorem}

Thus,  by the additivity of the Mabuchi constant for the product manifold (\cite[Proposition 2]{Y23}), Corollary \ref{noExtK2} implies the following 
\begin{corollary}\label{noMS}
 For each $n \geq 10$, there exists a toric Fano manifold of dimension $n$ admitting neither an extremal K\"ahler metric
 nor a Mabuchi soliton in the first Chern class. 
\end{corollary}

In fact, Corollary \ref{noMS} can alternatively be proven  by computing the Mabuchi constant using Proposition \ref{MabuchiTest}. See Subsection \ref{sec:MC} for our independent proof.  
 
We end the introduction by proposing the following natural question: 

\begin{question}
    Is there a toric Fano manifold of dimension less than $10$ that does not admit an extremal K\"ahler metric in the first Chern class? 
\end{question}
 
Throughout the paper, we assume that a {\em{manifold}} is a smooth irreducible complex projective variety, unless otherwise specified.

\vskip 7pt

\noindent{\bfseries Acknowledgements.}
DongSeon Hwang was supported by the National Research Foundation of Korea(NRF) grant funded by the Korea government(MSIT) (2021R1A2C1093787) and the Institute for Basic Science (IBS-R032-D1). H. Sato and N. Yotsutani was partially supported by JSPS KAKENHI JP24K06679 and JP22K03316,  respectively.
The authors would like to thank Satoshi Nakamura, Yee Yao, and Atsushi Ito for their interests in this work and helpful comments. In particular, S. Nakamura kindly let us know a nice application (Corollary \ref{noExtK2}) of our theorem.

\section{Preliminaries}
In this section, we recall the notion of relative K-polystability for a polarized toric manifold and the instability criterion of relative K-polystability given in \cite{YZ19}.

\subsection{Potential functions}
Let $(X,\omega)$ be a Fano manifold with a K\"ahler metric $\omega$.
Let $\mathrm{Aut}^0(X)$ be the identity component of the automorphism group.
Then the Lie algebra $\mathfrak h(X)$ of $\mathrm{Aut}^0(X)$ consists of holomorphic vector fields. 
As in \cite[Section 1]{ZZ08}, for any $v\in \mathfrak h(X)$, there is a unique smooth function $\theta_v(\omega): X\to \C$ such that
\[
\iota_v \omega=\sqrt{-1} ~ \overline{\partial}\theta_v(\omega)
\quad \text{and} \quad \int_X \theta_v(\omega)\omega^n=0.
\]
We call $\theta_v (\omega)$ the {\emph{potential function}} of $v$ with respect to $\omega$.

When $X$ is a toric Fano manifold with the moment polytope $P$ and $v$ is the extremal vector field of $X$, we can determine the potential function by the following conditions.
See \cite[Lemma 2.1]{YZ19} and \cite[Section 5.1, $(30)$]{Yao22a} for more details.
\begin{lemma}\label{lem:PotentialFun}
Let $X_P$ be the $n$-dimensional toric Fano manifold associated with an $n$-dimensional reflexive Delzant polytope $P$ in $M_\R$.
Then the potential function $\theta_P$ of the extremal vector field $v$, which is affine linear on $P$ and normalized by $\int_P\theta_P \,dv=0$, is independent of choice of $v$ and uniquely determined by the $n+1$ equations
\[
\mathcal L_P(1)=0,\quad \mathcal L_P(x_i)=0 \quad \text{for}\quad i=1, \dots, n.
\]
\end{lemma}
For the definition of the functional $\mathcal L_P(\cdot)$, see \eqref{eq:ModifiedDF}.
 
In Section \ref{sec:algorithm}, we give a good overview which explains how to compute the potential function systematically for a given toric Fano manifold.
The following lemma is useful in computing the potential function in practice. 

\begin{lemma}\cite{ZZ08}\label{lem:ZZ}
Let $P$ be an $n$-dimensional reflexive Delzant polytope in $M_{\R}$.
Then we have the equalities
\begin{equation}\label{eq:Stokes}
\int_{\partial P} x_i\,d\sigma=(n+1)\int_Px_i\,dv \qquad \text{and} \qquad \Vol(\partial P)=n\Vol(P).
\end{equation}
\end{lemma}

We provide the proof for the reader's convenience. 
\begin{proof}[Proof of Lemma \ref{lem:ZZ}]
Let $\mathcal F(P)=\set{F_1, \dots, F_d}$ be the set of facets of $P$.
Thus we have $\partial P=\bigcup_{k=1}^d F_k$. For $k=1,\dots, d$, let $P_k=\conv\set{0, F_k}$. 
Then $P$ can be written as the union of $P_i$: $P=\bigcup_{k=1}^d P_k$.
Let $\nu_k$ denote the outer normal vector of each facet $F_k$. Then the boundary measure $d\sigma$ on $\partial P$ defined in $\eqref{eq:BoundaryMes}$ is expressed as
$(\nu_k,\bs x)d\sigma_0$ on each facet $F_i$, where $d\sigma_0$ is the standard Lebesgue measure on $\partial P$.

Let $f$ be a rational piecewise linear function and let $f_j=\frac{\partial f}{\partial x_j}$. Using the Stoke's theorem, we see that
\begin{align*}
\int_{F_k}f\, d\sigma=\int_{\partial P_k}f\cdot(\nu_k,\bs x)\,d\sigma_0=\int_{P_k}\mathrm{div}(f\cdot \bs x)\,dv,
\end{align*}
where we used $(\nu_k,\bs x)=0$ for any $\bs x\in \partial P_k\setminus F_k$ in the first equality.
Since 
\[
\mathrm{div}(f\cdot \bs x)=\sum_{j=1}^n x_jf_j+nf, 
\]
we conclude that
\[
\int_{F_k}f\, d\sigma=\int_{P_k}\left(\sum_{j=1}^nx_jf_j+nf\right)dv.
\]
Hence, we have the equality
\begin{equation}\label{eq:IntBound}
\int_{\partial P}f\,d\sigma=\sum_{k=1}^d\int_{P_k}\left(\sum_{j=1}^nx_jf_j+nf\right)dv
\end{equation}
by taking the sum over all $k=1,\dots d$.

As a special case, we take a piecewise linear  function $f$ to be the coordinate function $x_i$, i.e., $f_i=1$ and $f_j=0$ for any $j\neq i$.
Then $\eqref{eq:IntBound}$ implies that
\[
\int_{\partial P} x_i\, d\sigma=\sum_{k=1}^d\int_{P_k}(x_i+nx_i)dv=(n+1)\int_Px_i\,dv.
\]
Similarly, if we take the constant function $f\equiv 1$, i.e., $f_j=0$ for all $1\leqslant j \leqslant d$, $\eqref{eq:IntBound}$ yields that
\[
\Vol(\partial P):=\int_{\partial P}1\, d\sigma =\sum_{k=1}^d\int_{P_k}n\, dv=n\cdot \Vol(P).
\]
The assertions are verified.
\end{proof}

\subsection{Relative K-stability}

In \cite{Don02}, Donaldson provided the condition for the K-polystability of a polarized toric manifold in terms of the positivity of a linear functional defined on $P$, which is called the {\em Donaldson-Futaki invariant}.
The Donaldson--Futaki invariant was generalized to the case of relative K-polystability in \cite{ZZ08}, and is given by
\begin{equation}\label{eq:ModifiedDF}
\mathcal L_P(f)=\int_{\partial P}f(\bs x)\, d\sigma -\int_P(\bar{S}+\theta_P(\bs x))f(\bs x)\,dv,
\end{equation}
for a convex function $f$. Here $\bar S$ is the average of the scalar curvature, $dv=dx_1\wedge \dots \wedge dx_n$ is the standard volume form on $M_\R$, and $d\sigma$ is the $(n-1)$-dimensional Lebesgue measure of $\partial P$ defined as follows:
let $\ell_j(\bs x)=\bracket{\bs x, u_j}+c_j$ be the defining affine function of the facet $F_j$
of $P$. On each facet $F_j=\Set{\bs x\in P | \ell_j(\bs x)=0}\subset \partial P$,
we define the $(n-1)$-dimensional Lebesgue measure $d\sigma_j$ by
\begin{equation}\label{eq:BoundaryMes}
dv=\pm d\sigma_j\wedge d\ell_j,
\end{equation}
up to the sign. We remark that 
\[\bar S=\frac{\Vol(\partial P)}{\Vol(P)}\]
by \cite[p.309]{Don02} and the functional
$\mathcal L_P$ corresponds to the modified Futaki invariant in \cite{Sz}.
Moreover, we recall that the potential function $\theta_P(\bs x)=\sum \alpha_ix_i+c$ of $P$ is uniquely determined by the $n+1$
equations
\begin{equation}\label{eq:n+1-eq}
\mathcal L_P(1)=0,\quad \mathcal L_P(x_i)=0 \quad \text{for}\quad i=1, \dots, n
\end{equation}
by Lemma \ref{lem:PotentialFun}.
See Section \ref{sec:algorithm} for the algorithm to compute the potential function of $P$. 
A convex function $f:P\to \R$ is said to be {\emph{rational piecewise linear}} if
$f$ is of the form
\[
f(\bs x)=\max\Set{f_1(\bs x), \dots, f_m(\bs x)}
\]
where each $f_k$ a rational affine function.
Moreover, a rational piecewise linear function $f$ is said to be {\emph{simple piecewise linear}} if it is of the form $f(\bs x)=\max\set{0,u(\bs x)}$ for a linear function $u(\bs x)$.
\begin{definition}\label{def:relK}\rm
A polarized toric manifold $(X_P, L_P)$ is called {\emph{relatively K-semistable}} (along toric degenerations) if $\mathcal L_P(f)\geqslant 0$ for any rational piecewise linear convex functions.
Moreover, it is called {\emph{relatively K-polystable}} if it is relatively K-semistable and the equality holds if and only if $f$ is affine linear.
\end{definition}
Now we consider the case where the toric manifold $X_P$ of dimension $n$ is Fano, and $L_P$ is the anti-canonical line bundle. Then we see $\bar S=n$ (cf.\cite[p.19]{Ti00}). 
By \cite[Theorem 0.1]{ZZ08} or  \cite[Theorem 1.4 (1)]{YZ19}, 
the sufficient condition of relative K-polystability for $(X_P, L_P)$ is given as
\begin{equation}\label{ineq:relK-sta}
\sup_{\bs x\in P}\theta_P(\bs x)\leqslant 1.
\end{equation}

By Atiyah-Guiilemin-Sternberg convexity theory, \cite{A82,GS82, GS84} 
$\eqref{ineq:relK-sta}$ is equivalent to the condition
\[
M_{X_P}:=\max_{\bs a\in \mathcal V(P)}\set{\theta_P(\bs a)}\leqslant 1.
\]
Here, the number $M_{X_P}$ is called the {\em Mabuchi constant} of the toric Fano variety $X_P$. Moreover, Yao proved the following. See also \cite[Proposition 1.2]{NSY23}.
\begin{proposition}[\cite{Yao22a}]\label{MabuchiTest}
    $X_P$ is relatively Ding unstable iff  $M_{X_P} > 1.$
\end{proposition}
 
Condition $\eqref{ineq:relK-sta}$ has been verified for all toric del Pezzo surfaces ($5$ classes) in \cite{ZZ08}, and for some of toric Fano $3$-folds ($7$ classes out of $18$ classes) in \cite{YZ19}.

\subsection{The instability criterion of relative K-stability}
\label{sec:relKsta}
Another important contribution of the work in \cite{YZ19} is that they provided the instability criterion of relative K-polystability in terms of the associated moment polytope.
More precisely, for the given moment polytope $P$ of an anti-canonically polarized toric Fano manifold, 
we consider the polytope 
\[ 
P^-=\set{\bs x\in P| 1-\theta_P(\bs x) \leqslant 0}, 
\]
 where $\theta_P(\bs x)=\sum\alpha_ix_i+c$ is the potential function of $P$ determined by $\eqref{eq:n+1-eq}$.
Then we have the following.
\begin{proposition}[Theorem 1.4 (2) in \cite{YZ19}]\label{prop:insta}
Let $P$ be the moment polytope of a toric Fano manifold $(X, -K_X)$ with the potential function $\theta_P(\bs x)=\sum\alpha_ix_i+c$. 
Assume that $\Vol(P^-) \neq 0$. If 
\begin{equation}\label{ineq:insta}
1-c<\frac{\int_{P^-}(1-\theta_P(\bs x))^2dv}{\Vol(P^-)},
\end{equation}
then there exists a simple piecewise linear function $f(\bs x)$ such that $\mathcal L_P(f)<0$.
In particular, the corresponding toric Fano manifold is relatively K-unstable.
\end{proposition}
Hence, in order to find our desired examples of a relatively K-unstable toric Fano manifold, it suffices to consider toric Fano manifolds satisfying $\eqref{ineq:insta}$.

\subsection{Algorithm for computing the potential function}\label{sec:algorithm} 
Based on the description in the previous subsections, we shall present an algorithm for computing the potential function $\theta_P(\bs x)$. See \cite[Section 4]{ZZ08} and \cite[p. $496$]{YZ19} for more details.

Let $X$ be an $n$-dimensional toric Fano manifold and $P$ be the corresponding $n$-dimensional moment polytope  in $M_\R \cong \R^n$. For the standard coordinates $x_1, \dots, x_n$ of $M_\R$, let $dv=dx_1\wedge \dots \wedge dx_n$ be the volume form of $M_\R$.
Then the functional 
\[\mathcal L_P(f)  = \int_{\partial{P}} u d\sigma - \int_P (\bar{S} + \theta_P)f(\bs x) dv\]
satisfies 
\[\mathcal L_P(1) = 0 \quad \text{and} \qquad \mathcal L_P(x_i) = 0 \quad \text{for} \quad i=1, \ldots n.\]

 We would like to compute the potential function 
\[ \theta_P(\bs x) =  
\sum_{j=1}^n a_jx_j+c.\]
For $1 \leq i,j \leq n$, let 
\[ b_i := \int_P x_i dv, \qquad c_{ij} := \int_P x_i x_j dv \]
for simplicity. Since $\mathcal L_P(x_i) =0$ and 
\[ \mathcal L_P(x_i) = \int_{\partial(P)} x_i d\sigma - \int_P (\bar{S} + \theta_P) x_i dv = (n+1)b_i - ( nb_i + \sum a_jc_{ij} + cb_i )           \]
by  the fact 
$\bar{S} = \frac{\vol(\partial(P)}{\vol(P)}$ and Lemma \ref{lem:ZZ}, 
we have 
\begin{equation}\label{eq:pot1}
\sum_{j=1}^n
\Big( \Big(c_{ij} - \frac{b_ib_j}{\vol(P)}\Big)a_j \Big) = b_i
\end{equation} 
for each $1 \leq i \leq n.$ Similarly, since $\mathcal L_P(1) = 0$ and 
\[  \mathcal L_P(1) = \int_{\partial(P)}  d\sigma - \int_P (\bar{S} + \theta_P)  dv = n \vol(P) - ( n \vol(P) + \sum a_j b_j + c \vol(P) ),    \]
we have 
\begin{equation}\label{eq:pot2}
c = -\sum_{j=1}^n    
\frac{a_jb_j}{\vol(P)}.
\end{equation} 
 
In summary we can compute the potential function in three steps.

\renewcommand{\labelenumi}{\emph{Step }$\arabic{enumi}$.}
\begin{step}\rm
Compute the volume $\vol(P)$ and the integrations $b_i$'s and $c_{ij}$'s. 
\end{step}

\begin{step}\rm
Compute the coefficients $a_1, a_2, \ldots, a_n$ of the linear terms of the potential function by solving the matrix equation in \eqref{eq:pot1}.
\end{step}

\begin{step}\rm 
Compute the constant term $c$ from the equation \eqref{eq:pot2}.
\end{step}

\section{Proof of the main theorem}\label{sec:main}

\subsection{Fano polytope} We first describe our toric Fano manifold $X$ of dimension $10$ quickly. 
For more details, see Example \ref{mainex}. There, 
one can see the precise geometric construction of $X$ and a generalization. 
\begin{example}\label{10dim}\rm
Put $M_\R:=\R^{10}$ and $N_\R:=\Hom_\R(M_\R,\R)\cong\R^{10}$. 
Let $\Delta$ be the convex hull of the $18$ elements 
\[
(1,0,0,0,0,0,0,0,0,0),
(0,1,0,0,0,0,0,0,0,0), 
(0,1,0,0,0,0,0,0,0,0),
\]
\[
(0,0,0,1,0,0,0,0,0,0),
(-1,-1,-1,-1,0,0,0,0,0,0),
\]
\[
(0,0,0,0,1,0,0,0,0,0),
(0,0,0,0,0,1,0,0,0,0),
(0,0,0,0,0,0,1,0,0,0),
\]
\[
(0,0,0,0,0,0,0,1,0,0),
(0,0,0,0,-1,-1,-1,-1,0,0),
\]
\[
(0,0,0,0,0,0,0,0,1,0),
(0,0,0,0,0,0,0,0,-1,0),
\]
\[
(0,0,0,0,0,0,0,0,0,1),
(0,0,0,0,0,0,0,0,0,-1)
\]
\[
(1,0,0,0,1,0,0,0,0,0),
(1,0,0,0,1,0,0,0,1,0),
\]
\[
(0,1,0,0,0,1,0,0,0,0)\mbox{ and }
(0,1,0,0,0,1,0,0,0,1)
\]
in $N_\R$. $\Delta$ is the $10$-dimensional {\em Fano polytope} whose 
vertices are these $18$ elements. Namely, $\Delta$ contains $\boldsymbol{0}=(0,0,0,0,0,0,0,0,0,0)$ 
in its interior, and for any facet $F\subset\Delta$, the set of vertices of $F$ is a $\Z$-basis for $\mathbb{Z}^{10} \subset \mathbb{R}^{10} \cong N_\R$. 
The $10$-dimensional toric manifold $X$ of Picard number $8$ associated to 
the normal fan $\Sigma$ constructed from $\Delta$ is a Fano manifold, and this is our main target. 
We remark that there exists a sequence 
\[
X\to Y_1 \to Z_2\to Z_1\to \mathbb{P}^4\times\mathbb{P}^4\times\mathbb{P}^1\times\mathbb{P}^1
\]
of blow-ups along torus invariant submanifolds of dimension $8$ 
(the notation is due to Example \ref{mainex}). It is well-known that 
the moment polytope $P\subset M_\R$ corresponding to $X$ is the dual polytope of $\Delta\subset N_\R$. 
\end{example}

\begin{remark}
The toric Fano manifold $X$ of dimension $10$  in Example \ref{10dim} 
is a generalization of the toric Fano manifold $X_{788}$ of dimension $5$ with ID:788 according to \cite{Paffenholz}. Recall that $X_{788}$ is a two-times blow-up of $\mathbb{P}^2\times\mathbb{P}^2\times\mathbb{P}^1$ 
along torus invariant submanifolds of codimension $2$. 
\end{remark}

\subsection{The potential function}\label{sec:PotFun}
Let $P$ be the moment polytope corresponding to $X$, which is a $10$-dimensional lattice polytope with $500$ vertices.
We refer the reader Subsection \ref{sec:VerticesP} for the specific data of all vertices. 
We compute the potential function $\theta_P$ following the algorithm given in Section \ref{sec:algorithm}. Through the computer calculations using \cite{Sage}, we obtain the following results.

\scriptsize 
\[
b_0 = \frac{1160242379}{907200}, b_1 =  b_2=b_5=b_6 = \frac{74830759}{302400}, 
b_3 = b_4=b_7=b_8=\frac{-74830759}{453600}, b_9 = b_{10} = \frac{70493741}{604800}. 
\]

\tiny
\[
c_{1,1} = c_{2,2} =c_{5,5} =c_{6,6} = \frac{178450577}{207360}, c_{1,2} = c_{5,6}  = \frac{-5546027789}{26611200},\]
\[ 
c_{1,3} = c_{1,4} = c_{2,3} = c_{2,4} = c_{5,7} = c_{5,8} =c_{6,7} =c_{6,8} = -\frac{26032694389}{119750400},
c_{1,5} = c_{2,6} = -\frac{8492713417}{39916800},
c_{1,6} = c_{2,5} = \frac{14971354001}{79833600},\]
\[ 
c_{1,7} = c_{1,8} =c_{2,7}=c_{2,8}=c_{3,5}=c_{3,6}=c_{4,5}=c_{4,6} =\frac{671357611}{79833600},
c_{1,9} = c_{2,10} = c_{5,9} = c_{6,10} = -\frac{1954923461}{53222400},\]
\[ 
c_{1,10} = c_{2,9} = c_{6,9}=c_{5,10}  =\frac{401887133}{11404800}, 
c_{3,3} = c_{4,4} = c_{7,7} = c_{8,8} = \frac{11836195861}{17107200},
c_{3,4} = c_{7,8}  =  -\frac{30787982249}{239500800},\]
\[ 
c_{3,7} = c_{3,8} = c_{4,7} = c_{4,8}   = -\frac{671357611}{119750400},
c_{3,9} = c_{3,10}=c_{4,9}=c_{4,10}=c_{7,9}=c_{7,10}=c_{8,9}=c_{8,10} = \frac{238350521}{479001600},\]
\[ 
c_{9,9} = c_{10,10} = \frac{47610261247}{119750400},
c_{9,10} = \frac{2238581}{268800}. 
\]
    
\normalsize

Thus we get   
\tiny
\begin{align*}\theta_P(\bs x) &= -\frac{6652648658253133533458927983168676127683718266602824699363990525289674109591035878412}{53176342041336824655798753434693514382090025600989171593542652235046808161762115363697}x_1 \\
&+\frac{1687745798567404193815209217652629451060504388870318301921068835922145585283243903476}
{53176342041336824655798753434693514382090025600989171593542652235046808161762115363697}x_2  \\
&-\frac{21147646246417011499866967397062900688335377230874174724893114467501367151381425063936}{53176342041336824655798753434693514382090025600989171593542652235046808161762115363697}x_3  \\
&- \frac{16449189574770853901820814565647695141892141945507143699918086624224225617813524184576}{53176342041336824655798753434693514382090025600989171593542652235046808161762115363697}x_4 \\
&+ \frac{18867213846247881068138780627879434410159166376320806192304262236814428074696514423284}{53176342041336824655798753434693514382090025600989171593542652235046808161762115363697} x_5 \\
&+ \frac{18022542054726735081290778189303857192199055260809176385339483335125600021972271934964}{53176342041336824655798753434693514382090025600989171593542652235046808161762115363697} x_6 \\
&-\frac{4698456671646157598046152831415205546443235285367031024975027843277141533567900879360}{53176342041336824655798753434693514382090025600989171593542652235046808161762115363697}x_7  \\
&+ \frac{18624673824124080312696733009635412309289343301014612819864978181445157138478256822904}{53176342041336824655798753434693514382090025600989171593542652235046808161762115363697} x_9\\
&+\frac{13250624276073817145781262775194444874240614651756716947770180446185628747942770170488} {53176342041336824655798753434693514382090025600989171593542652235046808161762115363697}x_{10}\\
&
-\frac{16867374143720575167184942526540793898738108171965667017693911650767893692338734579977015850328}{61697445596558353778949801830303224262546846295462581202931147867160375877959493561648257515163}.
\end{align*}

\normalsize 

\subsection{The Mabuchi constant}\label{sec:MC}
One can compute the Mabuchi constant
\scriptsize
\begin{align*} 
M_{X_P} &= \frac{151391597288670805729207671187119501031257600642015195734257750130579621051110153773967855027680}{61697445596558353778949801830303224262546846295462581202931147867160375877959493561648257515163} \\
& \approx 2.45377415263876
\end{align*}
\normalsize
where the value is attained at the vertex $(-1, 4, -1, -1, 4, -1, -1, -1, 1, 1)$ of the moment polytope $P$. Thus, by Proposition \ref{MabuchiTest},  the toric Fano variety $X_P$ is relatively Ding unstable, hence it does not admit a Mabuchi soliton  by \cite{Yao22a}. For the higher dimensional cases, consider a product $Y$ of $X_P$ with a suitable toric Fano manifold $X$. Then, by the additive property of the Mabuchi constant of the product toric manifold (\cite[Proposition 2]{Y23}), 
$$M_Y=M_{X_P}+M_X > 1. $$
This proves Corollary \ref{noMS}.

\subsection{The instability condition}
By using \cite{Sage}, we can compute the vertices of the polytope 
\[
P^- = \set{\bs x \in P : 1- \theta_P(\bs x) < 0 }, 
\] 
whose $346$ vertices are listed in Section \ref{sec:VerticesP-}. 
For the polytope $P^-$,  we would like to verify the inequality 
$$1 - c < \frac{\int_{P^-}^{} (1-\theta_P)^2 \,dv}{\Vol(P^-)}$$
where $c$ denotes the constant term of the potential function $\theta_P$. Thus 
\scriptsize
\begin{align*} 
1-c &= \frac{78564819740278928946134744356844018161284954467428248220625059517928269570298228141625273365491}{61697445596558353778949801830303224262546846295462581202931147867160375877959493561648257515163} \\
& \approx 1.27338853303615. 
\end{align*}
\normalsize

Unfortunately, SageMath does not compute the integrations over a polytope including  rational, but not integral, vertices. Instead, we use another computer algebra system LattE (\cite{LattE}) to compute the following integrations over the rational polytope $P^-$; 
\begin{align*}
\Vol(P^-)  &\approx 27.9812402670852, \\ 
\int_{P^-}^{} (1-\theta_P)^2 \,dv & \approx 73.7005491763169. 
\end{align*}

Thus we have 
\[ 
(1-c) - \frac{\int_{P^-}^{} (1-\theta_P)^2 \,dv}{\Vol(P^-)}  \approx -1.36053864363057 < 0,
\]
which completes the proof of Theorem \ref{thm:main} by Prop \ref{prop:insta}. 
See Section \ref{sec:ExactVal} for the exact values of the above computation.

\section{Further discussions and heuristics}\label{GeneralExample}
In fact, we have checked that no toric Fano manifold of dimension at most $5$ satisfies the instability condition in Proposition \ref{prop:insta}. However, among the $866$ toric Fano manifolds of dimension $5$, ID:788, according to the database `Smooth Reflexive Lattice Polytopes' by Paffenholz (\cite{Paffenholz}), seems to be the best candidate for our purpose. This leads us to consider the following example, generalizing the one with ID:788. 
 
\begin{example}[Toric Fano manifold of dimension $5r$]\label{mainex}\rm 

Let $r\in\N$. We put $M_{\R}:=\R^{5r}$ and $N_{\mathbb{R}}:=\Hom_{\R} (M_{\R},\R)\cong\R^{5r}$ 
as usulal.  
Let $\Sigma'$ be the fan in $N_\R$ associated to the $5r$-dimensional toric manifold 
\[
Z:=\mathbb{P}^{2r}\times\mathbb{P}^{2r}\times\left(\mathbb{P}^1\right)^r. 
\]
Let $\{e_1,\ldots,e_{5r}\}$ be the standard basis for $N_\R$. Then the rays of 
$\Sigma'$ are generated by 
\[
u_1:=e_1,\ \ldots,\ u_{2r}:=e_{2r},\ u_{2r+1}:=-(e_1+\cdots+e_{2r}),\ 
\]
\[
v_1:=e_{2r+1},\ \ldots,\ v_{2r}:=e_{4r},\ v_{2r+1}:=-(e_{2r+1}+\cdots+e_{4r}),\ 
\]
\[
w_{1,1}:=e_{4r+1},\ w_{1,2}:=-e_{4r+1},\ \ldots,\ w_{r,1}:=e_{5r},\ w_{r,2}:=-e_{5r},\ 
\]
and the maximal cones of $\Sigma'$ are 
\[
\langle \{u_1,\ldots,u_{2r+1},v_1,\ldots,v_{2r+1},w_{1,1},w_{1,2},\ldots,w_{r,1},w_{r,2}\}
\setminus\{u_i,v_j,w_{1,k_1},\ldots,w_{r,k_r}\}\rangle
\]
\[
(1\le i,j\le 2r+1,1\le k_1,\ldots,k_r\le 2),
\]
where $\langle U\rangle$ stands for the convex cone generated by $U$ 
for a subset $U\subset N_\R$. 

Let $\Sigma''$ be the fan obtained from $\Sigma'$ by the 
star subdivisions along the $r$ rays spanned by 
\[
y_1:=u_1+v_1,\ldots,y_r:=u_r+v_r. 
\]
Correspondingly, we have the sequence 
\[
Y:=Z_r\to Z_{r-1} \to \cdots \to Z_1 \to Z_0:=Z
\]
of toric morphisms, where $Z_i\to Z_{i-1}$ is the blow-up along 
the torus invariant submanifold $V(\langle u_i,v_i\rangle)\subset Z_{i-1}$ 
of codimension $2$ for $1\le i\le r$. 
Here, 
$Y$ is the toric manifold associated to $\Sigma''$, while 
$V(\sigma)$ stands for 
the torus invariant submanifold associated to the cone $\sigma$ in the fan. 

Next, we construct the fan $\Sigma$ from $\Sigma''$ by the 
star subdivisions along the $r$ rays spanned by 
\[
z_1:=w_{1,1}+y_1,\ldots,z_r:=w_{r,1}+y_r. 
\]
Let $\mathfrak{X}_r$ be the toric manifold associated to $\Sigma$. 
Then we obtain the sequence 
\[
\mathfrak{X}_r=Y_r\to Y_{r-1} \to \cdots \to Y_1 \to Y_0:=Y
\]
of the associated morphisms, 
where $Y_i\to Y_{i-1}$ is the blow-up along 
the torus invariant submanifold $V(\langle w_{i,1},y_i\rangle)\subset Y_{i-1}$ 
of codimension $2$ for $1\le i\le r$. 
One can check that $\mathfrak{X}_r$ is a toric Fano manifold of dimension $5r$ and of Picard number $3r+2$. 
The vertices of the Fano polytope $\Delta_r\subset N_\R$ associated to $\mathfrak{X}_r$ are  
\[
u_1,\ldots,u_{2r+1},v_1,\ldots,v_{2r+1},w_{1,1},w_{1,2},\ldots,w_{r,1},w_{r,2},
y_1,\ldots,y_r,z_1,\ldots,z_r,
\]
while the moment polytope in $M_\R$ corresponding to $\mathfrak{X}_r$ is the dual polytope of $\Delta_r$. 
\end{example}

\begin{remark}\rm
$\mathfrak{X}_1$ is nothing but the toric Fano manifold of dimension $5$ with ID:788 according to \cite{Paffenholz}, 
while $\mathfrak{X}_2$ is $X$ in Example \ref{10dim}.
\end{remark}

\begin{remark}\label{mainremark} \rm
By easy calculations, one can confirm that for $1\le i\le r$, there exists the relation 
\[
u_{i+1}+\cdots+u_{2r+1}+y_1+\cdots+y_i=v_1+\cdots+v_i
\]
among vertices of $\Delta_r$ which corresponds to the extremal contraction 
$\varphi_{R'_i}:Z_i\to \overline{Z}_{i}$ for an extremal ray $R'_i$ of the Kleiman-Mori cone 
$\NE(Z_i)$ of $Z_i$ 
by Reid's description of toric Mori theory (see \cite{reid}). 
Let $C'_i$ be the torus invariant curve on $Z_i$ which generates $R'_i$. 
Then the intersection number of a torus invariant divisor and $C'_i$ is calculated by this relation. 
In fact, we have 
\[
(-K_{Z_i}\cdot C'_i)=2r+1-i+i-i=2r+1-i.
\]
On the other hand, for $1\le i\le r$, 
there exists the relation 
\[
u_{r+1}+\cdots+u_{2r+1}+y_{i+1}+\cdots+y_r+z_1+\cdots+z_i=v_1+\cdots+v_r+w_{1,1}+\cdots+w_{i,1}
\] 
which corresponds to the extremal contraction 
$\varphi_{R''_i}:Y_i\to \overline{Y}_{i}$ for an extremal ray $R''_i\subset \NE(Y_i)$. 
As above, let $C''_i$ be the torus invariant curve on $Y_i$ which generates $R''_i$. 
Then we have 
\[
(-K_{Y_i}\cdot C''_i)=r+1+r-i+i-(r+i)=r+1-i. 
\]
Thus there exists a decreasing sequence 
\[
2r=(-K_{Z_1}\cdot C'_1)>\cdots > (-K_{Z_r}\cdot C'_r)>(-K_{Y_1}\cdot C''_1)>\cdots >(-K_{Y_r}\cdot C''_r)=1
\]
of $2r$ intersection numbers. This phenomenon indicates that $\mathfrak{X}_r=Y_r$ is {\em extreme} 
in the following sense. 
In Example \ref{mainex}, in order to construct our toric Fano manifold $\mathfrak{X}_r$ of 
dimension $5r$, we consider $2r$ times blow-ups of $Z$, though one can easily see that 
$Z$ can be blown-up more times. 
However, if we consider $2r+2$ times blow-ups of $Z$, 
then the resulting toric manifold is no longer a Fano manifold. 
This is what {\em extreme} means. Finally, one can also easily see that 
$Z$ has to be of dimension $5r$ to make the resulting manifold Fano. 
 \end{remark}

\section{Exact values of the computation}
In this section we present all the exact values of the computation in Section \ref{sec:main}. 
\subsection{Exact values of the computation}\label{sec:ExactVal}
Firstly, the volume of \( P^- \) is given by the rational number $a/b$ with  
{\footnotesize{
\begin{align*}
a=
&99138271978918682187254686405860132746921501576849130425934448578953424436052246010539704141\\
&79865731901786274145487623327683405301083902201907019266421496734753783360159739209760594677\\
&37135734999529691461156607159712565351709639115390236785045136758382009813065689840095311829\\
&42803966189234380402618968531640886197105405095400998327290864575354333473179431516482575495\\
&65297293465575244822545238100897452336386045695274845152646872937375634066718567646601359956\\
&17092297573522328191305979730516920767841213219340253076562038593537862720122497235197521749\\
&49945378841416285229924534687619715611846126301306744777842376510969950687271182585723242217\\
&17363226455667471939938619245059504209607929427034971669550827130769566271872407090694289385\\
&92107677840177291499388934051777315585657919069595558191231986557654141227821336892308630167\\
&82988995180094923058836295512784066845039832855888354650055796167223062035725032015000526392\\
&73275886444484792561770911651089616780491276218418550094688465095482367038130108939954968429\\
&03704977314127929463677109142327774819459059045275551588702139378208066397836472030539999055\\
&47948116321461617822246798398212722515871246643951368960957822625649945137902978963518985726\\
&79767723694154625312472160648188660666297686055026836082934606831584557134144469970252158467\\
&03387928622026285545005936551061402451714949017422332612863412517347080382380203081085256012\\
&02177838050702704286830756297441257384915383133785498989886121780309015200576986359663927025\\
&58951507207992579437683347227464493481530998949399853309626260725318374154548070039148240796\\
&98422077159733594726184393236784986077604271991968251678676612144013614452934608190450073699\\
&35440834594534665034045773368348895575297698548204704447558683121709239644946488044678417141\\
&37184900370648228192666630478184356108133905995993139381305317487365286200748330034639965513\\
&01676078929643303540746661338537700372224160707159885179295632209814739622109993570103432881\\
&15952981078444559529548803988936991586368993306744366340822225958275875538176003466158214529\\
&60770843116196673732028002469595112432989269732522997901022705468191066739017151978718756022\\
&29132447694380837243917495580960791157711526072949228194447832003704789638195841517615421827\\
&49436783403823590577854995998477677803692424759045857190168687524015024749986541948337959695\\
&55905658813466755534063920101854296532746142557781103676245035972320583412158550535428355565\\
&87110724661662886857251677395475885557440671136859291630140759663691829683466998735118846678\\
&822207233582526918606204511175964161126175418773994161, \qquad \text{and} 
\end{align*}

\vspace{-00.5395cm}

\begin{align*}
b=
&35430263645438377199723001608494763235307339917885584091791446454945934579009831125636781074\\
&16559917799136192206725653590744045353123673209890856047633982174328378240282520576981794340\\
&53083259289962019153840524421445450184453607955115433785255720387227333383266258654227902390\\
&59772715082226346519224440549347146780773181452663050439407305418206899315432831254484881465\\
&90895526706566177421021449288934556676116416791724092948625874752813453042101629324962810638\\
&41791574087112458710270908312196307385088644365456536867153853197681465369462618318813438498\\
&37710928678449470028080138562654779356041818671048641800664449445374046371303805176720991334\\
&05240369911374334033968662725164217719150147770963127209718755420910336612177992730319720202\\
&45972055218452113481992006895624556286988937741943177551933141768694676849742609431834168032\\
&09126283630490871453838482544354992488413390554384605443811633366256090573946491817635936641\\
&77758878030854192561776897496714944230304478507916422441611285621946916637967040940256581751\\
&98950986326461574306285436858386243674559441168869857643085599146467138769601716202688992337\\
&83916604500771689221348779458342627630603113242081015087434215144276800162450347232799148647\\
&11832878076705939813725735634026583803786091007449577862092386311971965306438790279556028567\\
&05790431003745559485328200882671365154978814789748701131403404889606981589138655544631642327\\
&58278251485872445929391557216227373646230939477637030635767791386251488540500911233315931548\\
&84542653030076236269744548691452375705920495694839449135025071658834954648272660140238187102\\
&62233223686861363469317519053163152210126504015140011112271718925192065741806389126427461478\\
&53026759017580937482797486075769374435234924870687479595013933329926899226247395997634988599\\
&01772762179862454824428742127046294654755761399274269350473094035720200663474123691446901614\\
&12922801254230990561145176438094804786609259235057322172115931878692304230796386075652489285\\
&27462682233609593257312928174864543437819587776047187541056257455731110430713170886433893597\\
&17583183489484129656448270597746481498318248937632827569500341658216568234274425594772469923\\
&23802423808451878868335815865535871240080239056664435258272914632964553268623907099014304610\\
&55061490260534225077781791869099057042958484484660428433072480968229046836127860424301702390\\
&62160373610170311061823262861401929821484066356224879443349207773531624496510416116304438958\\
&88950877437998126445841324521723814104060123812127398395486305783233882730107876536058926978\\
&77674415293889422636962268258173529623782427197440000.
\end{align*}
}}

Secondly, the exact value of the integration 
$\int_{P^-}^{} (1-\theta_P)^2 \,dv$ is the rational number $c/d$ with
{\footnotesize{
\begin{align*}
c=
&2175188099784750065997542609888467373337434176381577121622052613057913758329843008369240683\\
&2739103685577330848169752117568087790361095209379398945098057743576903623258657539530447246\\
&0658346407205770884580727448791281579650052677873381763679043925892475851912551182415026066\\
&3605278508581596390720882115158935964420928287540416767622563808173622512939274625005731408\\
&0601337931838972973663635966305183606251642587425935205942137713244130084438519753043945123\\
&4192799824892284568186087632119008657755486906360792360112749665457482189786362294518116102\\
&4063747835459273602155396128438863435483153255941287941497107628424099124097003302430367501\\
&1907102509228886371384313052009108947542389846969473951191319637071206276667814368378162863\\
&3010771722493235843747257180580180220692508444659297662559466837903322901369585992934859166\\
&8801544728090535245499415757535782675839187070784755969803556319722645151013304140109496976\\
&9071834863656933241772138116998919833862083461096545319965452098793557970479882857432662329\\
&6413066992206045617419428126602783935162572875504192706472580048825224490470094916862757129\\
&5769483377069509846305308987948509974875931048206874335417432789337143224205400654890783713\\
&2944989068626796136144141204692693331438382290358770479076458472625495771755840816685774043\\
&9210708514810678228059494427347459486677421290361172612852052686003891889295822747746785271\\
&8425480087237487270881822891390438652628363185696172687400260839504972152461185170530279217\\
&0219053450343478939552995315493589042368315080673911343748724881121596570493210427750994944\\
&2389655840129526346061546280087697241397889679752792225535942662820060709948486107198319201\\
&5387609298297396739204167810978036099612594298984473517122999968343135738067951350454390095\\
&9812614890553435218735997220172624300552460627230470972326997682954484419127748210484237737\\
&7858037990088241982824573011296875719986421237564601601351489303777590562671956480417948612\\
&9378121857739766138407343117751500019664828002020201331641682483424013417190753883586804175\\
&4240240826794509500538799654143357871292291515954562046759737907261640002492914890632307902\\
&8211939872003101177186091620633415357759558372224807706603627145757882330369776351093212376\\
&4937787078621875578198235063911640341396860836029361968059598103506084410215807302060909401\\
&0258814702710687852244335151505366586559850685301204495569184833339305143788201533798723765\\
&4591085000477208941701387363276721756122399327242667231448006621308101327890565477809369906\\
&1304629942903725787439626637374374915093817616169349188296798243714066067385093941115040525\\
&5830647789017625985843522211547670567705221170460199311540187328677941186215647216212059675\\
&9811348679137665910968899263102154692664408579885343029232140482146344704724132774288082012\\
&9959185913788884546951527605356684526632531839273711602641642193673269004397078627237755529\\
&8954545866002436416739304069133894162522480854940010784545363340431588910058686106059143674\\
&5419788180063416443232478702198787938913648096588184843808095578139911772526600606328712316\\
&2743755903179562132047438806203228922039170363286462274925415598397953026797245894085872631\\
&9671076817892272123943715921881857468769216644629139753853797958877605176353981813542302301\\
&3676944116742664522579607728297576815481246411803027048410710740687780574611962272427048751\\
&3243431815197995413497490562274838056605193864440533787137149299406366038550714271543982754\\
&1915506434472332626988537578242377475024370923112071406761830404396180407228091489852302487\\
&4970475718554079956314198585740769502746683702455406336860440568896407705229325051093780309\\
&0065031189579082976312433881910855421567796397060617868471829934766943759628585369316378381\\
&8934874711989915935770781898568961565346898569725496311437563176245338059836059352264300638\\
&1819141780892856122210710532781119071641511583752218807346320505761271234211331586643797165\\
&3351731533169296413166803939610900247975023383384886223204020494837905467608223211706895521\\
&6756546910833780342599776619531063231222800837546010873360006654805373099175031756426704734\\
&4612425191453526865094146567192009568407807356340551228891770093236140482226414789319151766\\
&7923420299719891783717221692544436863837986183298805377563565205622527003099771766872993299\\
&2276219003643769655300646248033690421841799128311013344960774295909499548379561501995985674\\
&1019059283567284685894463019860523108034825514676319977382555166538080143283473991245816629\\
&2634336868981858990084083708462529504147874235194412956369784944526213229356255517275107309\\
&0772363917167563717569817743127521551741559692567811084606419073260085326979099129090478354\\
&5928718702652598879492430712784666195258775194055003199538350160383335693234045768923512794\\
&7144967324887353278204429178078516727413455073020228489367760936190490454328289197384932126\\
&8059781489138177115167980592532404145000754751872811649402359312807361619343784275058406046\\
&3533062433805719605465470110181207035468379808980058412453307789451907662519133430679724582\\
&3454941034721458960661473875195804341236029862354852187937160776236115588596871931293967627\\
&175507526083785422277179635487280594471949394324786010503594600613879929519, \qquad \text{and}
\end{align*}
\begin{align*}
d=
&2951386555588556778320452593030052340610113893865354368463255643382209383934658042077287583\\
&8879284755734944314812932219589012793940330978903551693541204267055511218239806393132195621\\
&8230158694883133384880426786232804129698137920390207206742745443182015181257322293111621992\\
&8693360502564218995020923248468376796334689598122879924834033657381564979285201740976680539\\
&8163029447258664534725025547233356860783821577466021370553095731723396982071631401606099010\\
&0527367613964182398432360614906739671379002744862986736119381174593101714364997658009384385\\
&7249078017278604961001659077013863395319223226909094653427225524100580188617312053038071686\\
&4191275321845188637706827627524687875348174140607408071881891963781763248612930452598182517\\
&5918177868890246477756942604878232198670711038704319850979597691851559969154082685711355617\\
&3624806095888887203266760436324517831908727532361753073264041683014335998494750784284063172\\
&3619953400034176514953431306247175530267263299885672020639746858172092806965558653195745651\\
&3538783600517386437695991645399270450110533792030301611110476756580666206669413859910153872\\
&7274454379301568937231828644565443014229152851926568506477027345778920894607452187889431350\\
&5788868314740531852062328628266550823054804307610143627637667177683450040572575777886936631\\
&8450027109784432533739482710173831126751853523727715839690798371098949363090440656019327477\\
&4867469054307654077396851207245315782412966616509514306994801145002064168463470502198734044\\
&8367639699352132913254156180334402166584577954994323675897278519304931932423443530451650193\\
&2917657647118159335263059875683951773292321446843229604298892672248365970150416906859089509\\
&4347131280961096500849842919490211730818831549642352387822881876656858459164196833722742037\\
&8889151191799703753925763153821986058842818970501879470436080546110608203836071176628946154\\
&4146635464527623614231503010578647539552079155653966971663885231786896878703302506285151678\\
&9978070798593371797190715881294368256796115847209743073085757525792591459408068000190240305\\
&5564481896126885348697507259680021305199769038125301329482276584021387025805568243269949061\\
&0681222422793727972210854906322175851078491231354559031650528709574004251976996811263856414\\
&0536788051402330606089264456724508476205512690580327819485103743360954675625586347071462758\\
&6247333351357644822188625282175659573121211401095285345824604152570049553244223861463214723\\
&3464526629363660875420308086565013306690199735371674417494694142933800686310937460808469471\\
&8759538220930935810891125487782478870916071818659088167063516768057197208454522838829911444\\
&9425737891686253831649967829610489889530099496448680567112482276717415916554200223291539960\\
&4409526632772487837888504768718724497965053705276439728469828929374863657251744373276988000\\
&3888225415306132860858691026572011222761493098283642369009600304222489475037135243414994116\\
&1454750399661948699366652691092582589701435358145417858403569482962658763757209945952472984\\
&8445171900344776306080857228169519880491623406243408311512943697584543483429154906011432870\\
&3780208721583103546322415541743817087099544101819151958661530755994214172582759021069770287\\
&2950573836992893610786289946548826188201358791066604158660704498091841914533948578738215430\\
&5400333600730557345235394822735082194523322850148193231241218631539046480336568396873746901\\
&3597109063480971777387040213082334895498547063883971882881126771574908141434590828708483199\\
&0147441441788839019473397010551250147942712317152239614061184590399388107212877592774643266\\
&2166826474812586898247668211228190666126146845147692787266692053664911866982180306094854292\\
&8546139890459414169397775905146998280559309499383335410967865743880208263265120234530557472\\
&3862148683816380501383895142879641266401455304393262975135178889294632069928317332023893062\\
&3382043682734857577090682187782949046802870478756520487263824907362150217146488790338009287\\
&6641624949639797407884151811990022838242815548100375025348641605197969240349719114151218933\\
&2902114940566370818449244877307533915080381503819923792250326084516364426997200486896707064\\
&4797451543334734446993736913466721800016175877760015510952675151227649966789681073605907784\\
&6715028844566847915505685074092880235865356559668493583127223865376226206193804255051395652\\
&8961428160756116835655861982004625475414417110535277933946188022209196574590850226278251917\\
&7530930754222182393964611074083689749418714207389620408639590461427821688148101310610475431\\
&5245364816733489483143715491969346808205810358988924567467845642844438635343189663630502594\\
&7736886442551583472403364956436549584070019978057783675388027971599669735931299923153040982\\
&1527392813314147541382691414771184323362338198831049960571122404671577799263599784854223849\\
&7112303520422167633975532082878489161129896314749034705846305359202589036968972757092485403\\
&3635460276503991892416026381166925538153635894768715811716614472063639018948692538890472993\\
&4112081421743400239500961280326497445685428137543753869898854619607007085451308927418065759\\
&9052735951623165269751933436583761977055289512897556596809166361833915242619431440013367676\\
&4161754285379411429766611840526444070220304430251421966221447417298944000.
\end{align*}
}}

\vspace{-0.3cm}

Finally, the difference $(1-c) - \frac{\int_{P^-}^{} (1-\theta_P)^2 \,dv}{\Vol(P^-)}$ is the negative rational number 
$-e/f$ with
{\footnotesize{
\begin{align*}
e=
&5617899183279073902021141981550573922694912750923867870849273791478748834140163603890873130\\
&3414633458816076989108129896537401827864884018496362904093532255671711325167685368633677013\\
&8063084448776838044345855423857373698996530973666570996430979827518546486851165243971449707\\
&3679181815993876304674204544090608351472005471293356577071895563645081758240662780998598036\\
&8123550693360417476473756039845730664018938562409556235048649079685347976255953092203426832\\
&6123297754189947277032196088625118983478888208040464825377106457457450878749801372216259836\\
&0743718069044170198260073895086356971896042209466590295520964476085909067352721453001197097\\
&2827539514069014287545121427742030662408381174655202750720955805642176102170362027606688799\\
&1254762055669808216223227979975781076512256288582997066226186934900696918388245571837438110\\
&0003754070316496914210888061335451652574358349461477360180673348245639289633552812600661233\\
&6864498161137695430079822897660001054037748625337969819018466468844920520660949108164678047\\
&5914882984534361633605687170588286393262194129010972922606342716016036467927319490842677207\\
&8370121505623330473004781207753308241298264125185993734120344232651164766884827883038607004\\
&2436610495575625494060634424731512185626155522459769788879919065284999800107334079077851567\\
&8340802112826783269504693445500273891519151405242417175746152051999705249360166799322531654\\
&0400277078999339964758145461720801498585852428190235626583792600511810213926537398656341864\\
&4878622709400218237051979385744258812656607705517728982508523338708412960830591473522652726\\
&6989267278862156893064698114945126358677574940227899379452308573879507200308105720231678043\\
&2488019584477241988491645859695179551118726081267985967193319737958926411131864600727705169\\
&1736537203021654309737334237345498111413764780064519255432660850421952904716964043643480066\\
&7063869137903255666311644244465975946545098728366473464792230907544975972768084733708077404\\
&5994854510128436928789869744318972813966479016269616828832073661051674452399362568116757005\\
&6430736401036405681003573379287612249649101553320225158202154571838505587332015326422971014\\
&8974425115443957728945025144245293291820910372820365364179685305488625450394092273656164526\\
&3856142530944166806165579939408268151898270374330696706585460070083487242204902890244436218\\
&9800924662310339716502209138047672463664891565430748127506754999774419860992199857285168072\\
&0638550043125935263823805213604846753606515972673692621476177610955762995232465815268094821\\
&7298614721692711835959373683815769216728578426277476009720616763508085864938813677981742243\\
&5723017488234443664571949543637639669658530919183998151950557780082056810160640284365377212\\
&9530574544109621240208035269508030277127412682190275281552433792454576655378246880214345122\\
&9857205964304684901319407413697992982347697843242977760561021767442836138547642562373888777\\
&3279897475752069149572153365771515415793394945035126969489203160108800205011504377938586133\\
&9404804021525627434203094579988317861446564436888436890977349045922421454105776990866537815\\
&8406606338945211515805521459489950758808337803307799891712922137451243371148928207207162588\\
&9468401125385450710452547535783109776822806053795345348440410947494530819561812600774824976\\
&4261814448031734417916600776319926475205344530961924821151044590431614311645613260210049700\\
&3899466842988173234131095263393106759597468821326322308261321649575834187452614704632013470\\
&9672704961088290662851798844867031330858923443487470877584270724566169795407436900898873018\\
&3663867809897099845628250162709993169617134576722163321371433198586135485531135168501671791\\
&7842895072009009029219631938794022398187735480830761690168174832449136610711189955655905907\\
&7968007443625425120443392030751648931380572945091438510930817724450886635923747749991688579\\
&2319537027226149221430389625829257144616344168986736636755727109876696936564147093320149075\\
&1467952042777810402739382411194931464329674736981555422045872539046035703871633280962858562\\
&2278638510642392090143974585223278644336452795503666334538921789155522556119227512620507050\\
&9630480502453092285145633230838128972190046440525718828528470955505453013106876897830288781\\
&1180831008238532272740653633815894413076072425139914965126860485191722742844620846930783347\\
&6111514256975936894161972881681210712538021761894298516944245715677430426107681783860746533\\
&8447739521734944351333198190471830761063716566822839515344257836128194391291787881026589855\\
&6726201537016543959796900955136177209928003572180356216542899845592337415236147441714406653\\
&1762852435038277003953754372259337272377488586952742697254484386137399607269747177071605233\\
&0666955812916740609616784902965936533022638745441029075203862996903419564833734492295046181\\
&1448270486919159436793972858752557869649118023443867196585290814941231868328356859895234936\\
&3522015016915966334374116074470370925420506076131958792975922896959131785394815879116853099\\
&1432218228186451047511082306692757507007270103635523829251786832170109580171393699661065305\\
&0333685401142299764988301585267413150124174017804757159402170230352810832222124342693441887\\
&454342023322502098355141866530185184365590769838640246828245519786892051819, \qquad \text{and}
\end{align*}
\begin{align*}
f=
&4129172816648427646450290758485649112694633354612342928403934929043458145692389015211049235\\
&0767303500473814367585472570363925508195042900623756823962511521554984495717062457480752134\\
&2251164297880321776334267896094081772933051488421303967790129974937061603255420562291063470\\
&5377600401744043011633693433081375385106450995359768493762628645555528794479596040868806661\\
&6573812790587805155851357154834151532140532240751859368523462379761760984612683718434370143\\
&1156742231638448290219944428863092423623566854747175067565788062765866819863533691152751378\\
&1393462072824230485090978787527343650580587417857659358721754627982938416797920043418040622\\
&5712553599722589795005668393326145943975193690971663725775516014433412737860052878980542548\\
&2065728736220099606380447687915196637413939101879267856823858443644921440568998209530626384\\
&7010371143341566732966227038766817180014344076976481574002617792306940012602017428284596939\\
&9064813447352266278562866452626575444762392809803431137995448301945579565378701328310788118\\
&0014245365753874871233643092518946968022141077063576720761342958304534321391729332336424361\\
&9859924901953042767332990510099924533245811931096057362540613203239372209399958566615765356\\
&8481368897227577601696319968829357321649422427449397897874350166598763246531035353567667778\\
&9717527165699283761075812677298172104437338815965195240599589468373188603718266870008646722\\
&0762367035580924880549291924389301249982283608054762541129532585455891076366557453056065906\\
&0893963420724257269450774619732628296436538970036116727468716591721164385872277274116089507\\
&9007354178242659573453816713437601446146686456687700275810584480421101467238296390313906311\\
&4921803858665215815065872219130936181909933532259558094602030634482593995171050652231984312\\
&9620085227310537319601171297498726233914020462550664673397681760614517266453301670276853311\\
&9311981531251476135954714678651180434002110266555240034122097297216738813931575007587881801\\
&6000829111856222401763584301512814313741111903098971985684197430598916893239228217357016270\\
&8693816078262700121073827540062481114753875176114056065213575036868356124388086302718716568\\
&1899090012360950520248317807300502766415274953481790692635204976154262486814870457932608783\\
&3320418708979706891474468607125451665163800589239238067817786082763526050531842634565919075\\
&6003957804025362498583665128544906520625727433036827997513014824361416243933737766497551738\\
&1563405122963912683722101365785734545330839423632826492685152093007955877109434098607524546\\
&5493258099140526279641564311919208804212018361765647196894728247704797631593076528357281913\\
&2939164762138273145542656771150609506785595548056806570810537692917624960180263484688304532\\
&4287096733273641442138749565952867552489133169774344906343265806872806694883034941474728707\\
&3844196769029017998063971825080420668440837410300857803674693863397652922110239961197663670\\
&7129042684261881522535064163911691982767533178140029903551616636355568537476347427027844526\\
&5506875938415523226453497865309896462519275656071879589019231203073515136829817863202560617\\
&6618015016269518654838982960385664985266249609620198253189608451736427539515018157629866556\\
&6170950399599677406245890184063755423511466648203883404549292129635364925748605521740311870\\
&6213250062677237725894312948601454946143857267049373126870017198034325502833309765380649054\\
&7178715434079990180783075966299297612576967072657936294221308373566966416815085174712789254\\
&1080769131202717961622037741429816314197128992538694674341350679080724850610721173718911330\\
&8457308665184534305848544794054119743521705972594250248559037396953734723030227125423329469\\
&9007177557571403336045178704718601204128026023990530253069069441139538436805836242402444226\\
&3773469529013124918050633394068489120465513034221808672686846207832550267962341371198388588\\
&4962713092740976893110760241192931491285211109279719283048676941760771661305097156510070808\\
&2574337654606102161168107995934431478817087826145735194252856876671309010908556489575030712\\
&7039485329042148097541509405923739955052558260235788207240591136853853735530529834185577864\\
&2149946301901619179144550723577337417214134760858868278297115431939018034092559892978300545\\
&1857466851321812465280484590581395553893526791165410858234059256934375572450480655953064974\\
&3581210367996633368833733161015205127036256618665451712854817515154419368026809411184009077\\
&5667786493227763984164067668839681501958531668234322816759440190750288097814638919266289238\\
&6919239314263354796158587614069627121381071296523036733971970621092778074253719267663418689\\
&6862727481288801728149646787753897321653932846075573165645649273261273101386353374061563247\\
&3901365745628842927260106663827222990841789115551012502337739392668773530964234345760562175\\
&9236426280698684542223653078660647496487886776363675208509287701548420785791133254701449247\\
&5180228747517664564589732760650943028370877708393786306659326527080660936396868876014516916\\
&4855059552032557242072631016593399724773055419617663514484880868660725432629179240108509544\\
&0513687277134201486039120105311294645511615005562190841358669037986515835577904369004153511\\
&390822737919259505867013849302449393080245255791721336487667935695087239168.
\end{align*}
}} 

\subsection{The vertices of the moment polytope $P$}\label{sec:VerticesP}


\vspace{-0.3cm}

\subsection{ The vertices of the polytope $P^-$}\label{sec:VerticesP-}
In this subsection, we clarify $346$ vertices of $P^-$:

{\footnotesize{
\begin{flushleft}
\begin{tabular}{lll}
$(-1, -1, -1, -1, 0, 0, -1, -1, 1, 1)$,&
$(-1, -1, -1, -1, 0, 0, -1, 2, 1, 1)$,&
$(-1, -1, -1, -1, 0, 1, -1, -1, 1, -1)$,\\
$(-1, -1, -1, -1, 0, 1, -1, 1, 1, -1)$,&
$(-1, -1, -1, -1, 0, 3, -1, -1, 0, -1)$,&
$(-1, -1, -1, -1, 0, 3, -1, -1, 0, 1)$,\\
$(-1, -1, -1, -1, 0, 3, -1, -1, 1, -1)$,&
$(-1, -1, -1, -1, 0, 3, -1, -1, 1, 1)$,&
$(-1, -1, -1, -1, 1, 2, -1, -1, -1, -1)$,\\
$(-1, -1, -1, -1, 1, 2, -1, -1, -1, 1)$,&
$(-1, -1, -1, -1, 2, 1, -1, -1, -1, -1)$,&
$(-1, -1, -1, -1, 2, 1, -1, -1, 1, -1)$,\\
$(-1, -1, -1, -1, 3, 0, -1, -1, -1, 0)$,&
$(-1, -1, -1, -1, 3, 0, -1, -1, -1, 1)$,&
$(-1, -1, -1, -1, 3, 0, -1, -1, 1, 0)$,\\
$(-1, -1, -1, -1, 3, 0, -1, -1, 1, 1)$,&
$(-1, 0, -1, -1, 4, -1, -1, -1, -1, 0)$,&
$(-1, 0, -1, -1, 4, -1, -1, -1, -1, 1)$,\\
$(-1, 0, -1, -1, 4, -1, -1, -1, 1, 0)$,&
$(-1, 0, -1, -1, 4, -1, -1, -1, 1, 1)$,&
$(-1, 0, -1, 3, 4, -1, -1, -1, 1, 1)$,\\
$(-1, 1, -1, -1, 4, -1, -1, -1, -1, -1)$,&
$(-1, 1, -1, -1, 4, -1, -1, -1, 1, -1)$,&
$(-1, 4, -1, -1, 0, -1, -1, -1, 1, 1)$,\\
\end{tabular}
\end{flushleft} 
\vskip 20pt

{\hspace{-1.0cm}{
\begin{tabularx}{16cm}{X}
$(-\frac{29435432804629330289525928463486779095945079205335549493876718614396543887100696255123942482215}
{31567420184452255222958198068903261043785701691411118250175206039890055416610262077528175384188}$, \\
 \\ \vspace{-0.6cm} 
$\frac{124137693357986095958400522670196562227302184279568904244402336734066710136931482487708468634779}
{31567420184452255222958198068903261043785701691411118250175206039890055416610262077528175384188}, ~ -1, ~ -1$,\\
 \\ \vspace{-0.9cm} 
$-\frac{2131987379822924933432269605416481947840622486075568756298487425493511529509565822404232901973}
{31567420184452255222958198068903261043785701691411118250175206039890055416610262077528175384188}, ~ -1, ~ -1, ~  -1, ~ 1, ~ 1), $ 
\end{tabularx}
}}

{\hspace{-1.0cm}{
\begin{tabularx}{16cm}{X}
$(-\frac{29435432804629330289525928463486779095945079205335549493876718614396543887100696255123942482215}
{31567420184452255222958198068903261043785701691411118250175206039890055416610262077528175384188}$, \\ 
 \\ \vspace{-0.6cm}
$\frac{124137693357986095958400522670196562227302184279568904244402336734066710136931482487708468634779}
{31567420184452255222958198068903261043785701691411118250175206039890055416610262077528175384188}, ~ -1, ~ -1$, \\
$-\frac{2131987379822924933432269605416481947840622486075568756298487425493511529509565822404232901973}
{31567420184452255222958198068903261043785701691411118250175206039890055416610262077528175384188}, ~ -1, ~ -1$, \\
 \\ \vspace{-0.6cm}
$\frac{96834247933179690602306863812126265079197727560308923506824105545163677779340352054988759054537}
{31567420184452255222958198068903261043785701691411118250175206039890055416610262077528175384188}, ~ 1,  ~ 1)$ \\
\end{tabularx}
}}

\vskip 1pt

{\hspace{-1.0cm}{
\begin{tabularx}{16cm}{X}
$(-1,~  -1, ~  -1, ~ -1 $,  
$\frac{3102512192376358142679819001580668439063486709837887008636819750620191848141353615683920698011}
{27341889624411541491972394050768712725196827261070249408340863429956543260136803271783679550076}$, \\
 \\ \vspace{-0.6cm}
$0, ~ \frac{51581267056446724841264969099956757011330167812302611808044907109292894672132252927883438402141}
{27341889624411541491972394050768712725196827261070249408340863429956543260136803271783679550076}, ~ -1, ~ 1, ~ 1)$, \\
\end{tabularx}
}}

\vskip 1pt

{\hspace{-1.0cm}{
\begin{tabularx}{16cm}{X}
$(-1, ~ -1, ~  -1, ~ -1$,  
$\frac{7488518243447579868958738283374174613477473420486134371117067975051492313867084916435983083119}
{27341889624411541491972394050768712725196827261070249408340863429956543260136803271783679550076 }$, \\
 \\ \vspace{-0.6cm}
$1$, ~ $\frac{19853371380963961623013655767394538111719353840584115037223795454905050946269718355347696466957}
{27341889624411541491972394050768712725196827261070249408340863429956543260136803271783679550076}$, ~ $-1,~ 1, ~ -1)$, 
\end{tabularx}
}}

\vskip 1pt

{\hspace{-1cm}{
\begin{tabularx}{12cm}{X}
$(-1, ~ -1, ~ -1, ~ -1$, 
$\frac{18476448025683396527512061137496413347536773271814356809347312656761887222265264634141153408963}
{27341889624411541491972394050768712725196827261070249408340863429956543260136803271783679550076}, ~ 0 $\\
 \\ \vspace{-0.6cm}
$\frac{36207331223139686456432726964041012102856881250326142007334414203151199298008341909426205691189}
{27341889624411541491972394050768712725196827261070249408340863429956543260136803271783679550076}, ~ -1, ~ 1, ~ 0 )$, 
\end{tabularx}
}}

\vskip 1pt

{\hspace{-1cm}{
\begin{tabularx}{12cm}{X}
$(-1, ~ -1, ~ -1, ~ -1$, 
$\frac{46320783923977859209060461826835109839814989651671701992404901238666237400895043887642157995243}
{27341889624411541491972394050768712725196827261070249408340863429956543260136803271783679550076}, $\\ 
\\ \vspace{-0.6cm}
$0, ~\frac{8362995324845223774884326274702315610578664870468796824276825621246849119378562655925201104909}
{27341889624411541491972394050768712725196827261070249408340863429956543260136803271783679550076}, ~ -1, ~ -1, ~ 1)$,  
\end{tabularx}
}}

\vskip 1pt

{\hspace{-1cm}{
\begin{tabularx}{12cm}{X}
$(-1, ~ -1, ~ -1, ~ -1$, 
$\frac{50706789975049080935339381108628616014228976362319949354885149463097537866620775188394220380351}
{27341889624411541491972394050768712725196827261070249408340863429956543260136803271783679550076}$,\\ 
\\ \vspace{-0.6cm}
$1$, ~ $\frac{-23364900350637539443366987057859903289032149101249699946544286033140994606483971916610540830275}
{27341889624411541491972394050768712725196827261070249408340863429956543260136803271783679550076}, ~ -1, ~ -1, ~ -1)$,  
\end{tabularx}
}}

\vskip 1pt

{\hspace{-1cm}{
\begin{tabularx}{12cm}{X}
$(-1, ~ -1, ~ -1, ~ -1 $,  
$\frac{61694719757284897593892703962750854748288276213648171793115394144807932775018954906099390706195}
{27341889624411541491972394050768712725196827261070249408340863429956543260136803271783679550076}$,\\ 
\\ \vspace{-0.6cm}
$0$, ~ $-\frac{7010940508461814609947915861213429297894621691507672976433667284894846254745348362532031606043}
{27341889624411541491972394050768712725196827261070249408340863429956543260136803271783679550076}$, ~ $ -1, ~ -1, ~ 0)$, 
\end{tabularx}
}}

\vskip 1pt

{\hspace{-1cm}{
\begin{tabularx}{12cm}{X}
$(-1, ~ 0, ~ -1, ~ -1$,  
$\frac{27506183607442113352378648562546636491692169438059740018469416162059211278010066056440083125403}
{27341889624411541491972394050768712725196827261070249408340863429956543260136803271783679550076}$, \\
\\ \vspace{-0.6cm} 
$-1$, ~ $\frac{54519485265792511123538533589759501683898312345151008206553174127810418502400343758910955524825}
{27341889624411541491972394050768712725196827261070249408340863429956543260136803271783679550076}, ~  -1, ~ 1, ~ 1)$,  
\end{tabularx}
}}

\vskip 1pt

{\hspace{-1cm}{
\begin{tabularx}{12cm}{X}
$(-1, ~ 0, ~ -1, ~ -1$,  
$\frac{42880119440749151737210890698462381400165456000036209819179909068200906652133977074897315836355}
{27341889624411541491972394050768712725196827261070249408340863429956543260136803271783679550076}$, \\
\\ \vspace{-0.6cm} 
$-1$, ~ $\frac{39145549432485472738706291453843756775425025783174538405842681221668723128276432740453722813873}{27341889624411541491972394050768712725196827261070249408340863429956543260136803271783679550076}, ~  -1, ~ 1, ~ 0)$,  
\end{tabularx}
}}

\vskip 1pt

{\hspace{-1cm}{
\begin{tabularx}{12cm}{X}
$(-1, ~ 0, ~ -1, ~ -1 $, 
$\frac{70724455339043614418759291387801077892443672379893555002237497650105256830763756328398320422635}
{27341889624411541491972394050768712725196827261070249408340863429956543260136803271783679550076}$ \\
\\ \vspace{-0.6cm} 
$-1$, ~ $\frac{11301213534191010057157890764505060283146809403317193222785092639764372949646653486952718227593}
{27341889624411541491972394050768712725196827261070249408340863429956543260136803271783679550076}$, ~ $ -1, ~ -1, ~ 1)$, 
\end{tabularx}
}}

\vskip 1pt

{\hspace{-1cm}{
\begin{tabularx}{12cm}{X}
$(-1, ~ 0, ~ -1, ~ -1$, 
$\frac{86098391172350652803591533523716822800916958941870024802947990556246952204887667346855553133587}
{27341889624411541491972394050768712725196827261070249408340863429956543260136803271783679550076}$,\\
\\ \vspace{-0.6cm} 
$-1$, ~ $-\frac{4072722299116028327674351371410684625326477158659276577925400266377322424477257531504514483359}
{27341889624411541491972394050768712725196827261070249408340863429956543260136803271783679550076}$, ~  $-1, ~ -1, ~ 0)$, 
\end{tabularx}
}}

\vskip 1pt

{\hspace{-1cm}{
\begin{tabularx}{12cm}{X}
$(-1, ~ 0, ~ -1, ~ 3$, 
$\frac{103846370986858648996018705856006370440874894767103327090020707882131189518988898369586056510619}
{27341889624411541491972394050768712725196827261070249408340863429956543260136803271783679550076}$, \\
\\ \vspace{-0.6cm} 
$-1$, ~ $-\frac{21820702113624024520101523703700232265284412983892578864998117592261559738578488554235017860391}
{27341889624411541491972394050768712725196827261070249408340863429956543260136803271783679550076}$, ~ $ -1, ~ 1, ~ 1)$, 
\end{tabularx}
}}

\vskip 1pt

{\hspace{-1cm}{
\begin{tabularx}{12cm}{X}
$(-1,  ~1,  ~ -1, ~ -1$,  
$\frac{56295861073579090288356397405306110718734838876929840390782260797929531173604509797948307937903}
{27341889624411541491972394050768712725196827261070249408340863429956543260136803271783679550076}$, \\
\\ \vspace{-0.6cm} 
$-1$, ~ $\frac{25729807799655534187560784747000027456855642906280907834240329491940098606805900017402730712325}
{27341889624411541491972394050768712725196827261070249408340863429956543260136803271783679550076}$, ~ $ -1, ~ 1, ~ -1)$, 
\end{tabularx}
}}

\vskip 1pt

{\hspace{-1cm}{
\begin{tabularx}{12cm}{X}
$(-1,  ~ 1, ~ -1, ~ -1$, 
$\frac{99514132805180591354737040230560552119486341818763655374550342285975576726358200069906545235135}
{27341889624411541491972394050768712725196827261070249408340863429956543260136803271783679550076}$,\\
\\ \vspace{-0.6cm} 
$-1$, ~ $-\frac{17488463931945966878819858078254413943895860035552907149527751996105946945947790254555506584907}
{27341889624411541491972394050768712725196827261070249408340863429956543260136803271783679550076}$, ~ $ -1, ~ -1, ~ -1)$, 
\end{tabularx}
}}

\vskip 1pt

{\hspace{-1cm}{
\begin{tabularx}{12cm}{X}
$(-1, ~ 1, ~ -1, ~ -1$, 
$\frac{99514132805180591354737040230560552119486341818763655374550342285975576726358200069906545235135}
{27341889624411541491972394050768712725196827261070249408340863429956543260136803271783679550076}$, \\
\\ \vspace{-0.6cm} 
$-1$, ~ $-\frac{17488463931945966878819858078254413943895860035552907149527751996105946945947790254555506584907}
{27341889624411541491972394050768712725196827261070249408340863429956543260136803271783679550076}$, ~ $ -1, ~ -1, ~ -1)$, 
\end{tabularx}
}}

\vskip 1pt

{\hspace{-1cm}{
\begin{tabularx}{12cm}{X}
$(-1, ~ 4, ~ -1, ~ -1$, 
$\frac{19673406805533714017631706846258574132076554697728383102036851456406927867396552874815120687787}
{27341889624411541491972394050768712725196827261070249408340863429956543260136803271783679550076}$,\\
\\ \vspace{-0.6cm} 
$-1$, ~ $\frac{62352262067700910458285475306047564043513927085482365122985738833462701913013856940535917962441}
{27341889624411541491972394050768712725196827261070249408340863429956543260136803271783679550076}$, ~ $ -1, ~ 1, ~ 1)$, 
\end{tabularx}
}}

\vskip 1pt

{\hspace{-1cm}{
\begin{tabularx}{12cm}{X}
$(-1, ~ 4, ~ -1, ~ -1$, 
$\frac{50421278472147790787296191118090063949023127821681322703457837268690318615644374911729586109691}
{27341889624411541491972394050768712725196827261070249408340863429956543260136803271783679550076}$,\\
\\ \vspace{-0.6cm} 
$-1$, ~ $\frac{31604390401086833688620991034216074226567353961529425521564753021179311164766034903621452540537}
{27341889624411541491972394050768712725196827261070249408340863429956543260136803271783679550076}$, ~ $ -1, ~ 1, ~ -1)$, 
\end{tabularx}
}}

\vskip 1pt

{\hspace{-1cm}{
\begin{tabularx}{12cm}{X}
$(-1, ~ 4, ~ -1, ~ -1$, 
$\frac{62891678537135215084012349671513015532828057639562198085804932944452973420150243146773357985019}
{27341889624411541491972394050768712725196827261070249408340863429956543260136803271783679550076}$, \\
\\ \vspace{-0.6cm} 
$-1$, ~ $\frac{19133990336099409391904832480793122642762424143648550139217657345416656360260166668577680665209}
{27341889624411541491972394050768712725196827261070249408340863429956543260136803271783679550076}$, ~ $ -1, -1, 1)$,   
\end{tabularx}
}}

\vskip 1pt

{\hspace{-1cm}{
\begin{tabularx}{12cm}{X}
$(-1, ~ 4, ~ -1, ~ -1$,  
$\frac{93639550203749291853676833943344505349774630763515137687225918756736364168398065183687823406923}
{27341889624411541491972394050768712725196827261070249408340863429956543260136803271783679550076}$, \\
\\ \vspace{-0.6cm} 
$-1$, ~ $-\frac{11613881330514667377759651791038367174184148980304389462203328466866734387987655368336784756695}
{27341889624411541491972394050768712725196827261070249408340863429956543260136803271783679550076}$, ~  $ -1, ~ -1, ~ -1)$, 
\end{tabularx}
}}

\vskip 1pt

{\hspace{-1.5cm}{
\begin{tabularx}{12cm}{X}
$(-1, ~ \frac{16284206674569121979193106104606575620004762906626281314992593044205455875347803979670691900075}
{26494589591670393482362743865355713097178879313294723961579798826906175262124616047997572353148}$, \\
\\ \vspace{-0.6cm} 
$\frac{63199562100442058467895125491460563671531875033257890569746803436513069911026044164322025159369}
{26494589591670393482362743865355713097178879313294723961579798826906175262124616047997572353148}$,  ~ $ -1, ~ 4, ~ -1, ~ -1, ~ -1, ~ 1, ~ 1)$, 
\end{tabularx}
}}

\vskip 1pt

{\hspace{-1.5cm}{
\begin{tabularx}{12cm}{X}
$(-1, ~ \frac{20204302710043867773540422347529491950661726297688509992593096412623267785806654122155797953195}
{26494589591670393482362743865355713097178879313294723961579798826906175262124616047997572353148}$,\\
\\ \vspace{-0.6cm} 
$\frac{59279466064967312673547809248537647340874911642195661892146300068095258000567194021836919106249}
{26494589591670393482362743865355713097178879313294723961579798826906175262124616047997572353148}$,  
~ $ -1, ~ 0, ~ 3, ~ -1, ~ -1, ~ 1, ~ 1)$,
\end{tabularx}
}}

\vskip 1pt

{\hspace{-1.5cm}{
\begin{tabularx}{12cm}{X}
$(-1,  ~ \frac{47032078341183198748857590376438065436951336030579220916413578856488846623595626016585157321979}
{26494589591670393482362743865355713097178879313294723961579798826906175262124616047997572353148}$, \\
\\ \vspace{-.6cm} 
$\frac{32451690433827981698230641219629073854585301909304950968325817624229679162778222127407559737465}
{26494589591670393482362743865355713097178879313294723961579798826906175262124616047997572353148}$,  
~ $ -1, ~ 4, ~ -1, ~ -1, ~ -1, ~ 1, ~ -1)$,
\end{tabularx}
}}

\vskip 1pt

{\hspace{-1.5cm}{
\begin{tabularx}{12cm}{X}
$(-1, ~ \frac{50952174376657944543204906619360981767608299421641449594014082224906658534054476159070263375099}
{26494589591670393482362743865355713097178879313294723961579798826906175262124616047997572353148}$,\\
\\ \vspace{-0.6cm} 
$\frac{28531594398353235903883324976706157523928338518242722290725314255811867252319371984922453684345}
{26494589591670393482362743865355713097178879313294723961579798826906175262124616047997572353148}$, 
~ $ -1, ~ 0, ~ 3, ~ -1, ~ -1, ~ 1, ~ -1)$,
\end{tabularx}
}}

\vskip 1pt

{\hspace{-1.5cm}{
\begin{tabularx}{12cm}{X}
$(-1, ~ \frac{59502478406170623045573748929861017020756265848460096298760674532251501428101494251628929197307}
{26494589591670393482362743865355713097178879313294723961579798826906175262124616047997572353148}$,\\
\\ \vspace{-0.6cm} 
$\frac{19981290368840557401514482666206122270780372091424075585978721948467024358272353892363787862137}
{26494589591670393482362743865355713097178879313294723961579798826906175262124616047997572353148}$, ~ $ -1, ~ 4, ~ -1, ~ -1, ~ -1, ~ -1, ~ 1)$,  
\end{tabularx}
}}

\vskip 1pt

{\hspace{-1.5cm}{
\begin{tabularx}{12cm}{X}
$(-1, ~ \frac{62442550432776682391334236112053204268748988391756767806961052058564860360945631858492758737147}
{26494589591670393482362743865355713097178879313294723961579798826906175262124616047997572353148}$,\\ 
\\ \vspace{-0.6cm}
$\frac{17041218342234498055753995484013935022787649548127404077778344422153665425428216285499958322297}
{26494589591670393482362743865355713097178879313294723961579798826906175262124616047997572353148}$,~ $ -1, ~ 1, ~ 2, ~ -1, ~ -1, ~ -1, ~ 1)$, 
\end{tabularx}
}}

\vskip 1pt

{\hspace{-1.5cm}{
\begin{tabularx}{12cm}{X}
$(-1, ~ \frac{62442550432776682391334236112053204268748988391756767806961052058564860360945631858492758737147}
{26494589591670393482362743865355713097178879313294723961579798826906175262124616047997572353148}$, \\ 
\\ \vspace{-0.6cm}
$\frac{17041218342234498055753995484013935022787649548127404077778344422153665425428216285499958322297}
{26494589591670393482362743865355713097178879313294723961579798826906175262124616047997572353148}$,~ $-1, ~ 1, ~ 2, ~ -1, ~ -1, ~ -1, ~ 1)$,\\ 
\end{tabularx}
}}

\vskip 1pt

{\hspace{-1.5cm}{
\begin{tabularx}{12cm}{X}
$(-1, ~ \frac{72561310242458695076395228031988202467984050892558357085898122968929681310431321295049382023715}
{26494589591670393482362743865355713097178879313294723961579798826906175262124616047997572353148}$,\\
\\ \vspace{-0.6cm}
$\frac{6922458532552485370693003564078936823552587047325814798841273511788844475942526848943335035729}
{26494589591670393482362743865355713097178879313294723961579798826906175262124616047997572353148}$,~ $ -1, ~ 0, ~ 3, ~ -1, ~ -1, ~ 0, ~ -1)$, 
\end{tabularx}
}}

\vskip 1pt

{\hspace{-1.5cm}{
\begin{tabularx}{12cm}{X}
$(-1, ~\frac{72561310242458695076395228031988202467984050892558357085898122968929681310431321295049382023715}
{26494589591670393482362743865355713097178879313294723961579798826906175262124616047997572353148}$,\\ 
\\ \vspace{-0.6cm}
$\frac{6922458532552485370693003564078936823552587047325814798841273511788844475942526848943335035729}
{26494589591670393482362743865355713097178879313294723961579798826906175262124616047997572353148}$,~ $ -1, ~ 0, ~ 3, ~ -1, ~ -1, ~ 0, ~ -1)$, 
\end{tabularx}
}}

\vskip 1pt

{\hspace{-1.5cm}{
\begin{tabularx}{12cm}{X}
$(-1, ~ \frac{90250350072784699815238233201692506837702838972413035900181660344534892176349316288543394619211}
{26494589591670393482362743865355713097178879313294723961579798826906175262124616047997572353148}$,\\ 
\\ \vspace{-0.6cm}
$-\frac{10766581297773519368150001605625367546166201032528864015442263863816366389975468144550677559767}
{26494589591670393482362743865355713097178879313294723961579798826906175262124616047997572353148}$,~ $ -1,  4,  -1,  -1,  -1,  -1,  -1)$, 
\end{tabularx}
}}

\vskip 1pt

{\hspace{-1.5cm}{
\begin{tabularx}{12cm}{X}
$(-1, ~\frac{93190422099390759160998720383884694085695561515709707408382037870848251109193453895407224159051}
{26494589591670393482362743865355713097178879313294723961579798826906175262124616047997572353148}$, \\ 
\\ \vspace{-0.6cm}
$-\frac{13706653324379578713910488787817554794158923575825535523642641390129725322819605751414507099607}
{26494589591670393482362743865355713097178879313294723961579798826906175262124616047997572353148}$,  ~ $ -1, 1, 2, -1, -1, -1, -1)$, 
\end{tabularx}
}}

\vskip 1pt

{\hspace{-1.5cm}{
\begin{tabularx}{12cm}{X}
$(-1, ~\frac{103846370986858648996018705856006370440874894767103327090020707882131189518988898369586056510619}
{26494589591670393482362743865355713097178879313294723961579798826906175262124616047997572353148}$,\\ 
\\ \vspace{-0.6cm}
$-\frac{24362602211847468548930474259939231149338256827219155205281311401412663732615050225593339451175}
{26494589591670393482362743865355713097178879313294723961579798826906175262124616047997572353148}$, ~ $ -1,  0,  -1,  -1,  -1,  1,  1)$, 
\end{tabularx}
}}

\vskip 1pt

{\hspace{-1.5cm}{
\begin{tabularx}{12cm}{X}
$(-1, ~\frac{103846370986858648996018705856006370440874894767103327090020707882131189518988898369586056510619}
{26494589591670393482362743865355713097178879313294723961579798826906175262124616047997572353148}$,\\ 
\\ \vspace{-0.6cm}
$-\frac{24362602211847468548930474259939231149338256827219155205281311401412663732615050225593339451175}
{26494589591670393482362743865355713097178879313294723961579798826906175262124616047997572353148}$, ~ $ -1,  0,  -1,  -1,  -1,  1,  1)$, 
\end{tabularx}
}}

\vskip 1pt

{\hspace{-1.5cm}{
\begin{tabularx}{12cm}{X}
$(-1, ~\frac{103846370986858648996018705856006370440874894767103327090020707882131189518988898369586056510619}
{26494589591670393482362743865355713097178879313294723961579798826906175262124616047997572353148}$,\\ 
\\ \vspace{-0.6cm}
$-\frac{24362602211847468548930474259939231149338256827219155205281311401412663732615050225593339451175}
{26494589591670393482362743865355713097178879313294723961579798826906175262124616047997572353148}$, ~ $ -1, ~ 0, ~ -1, ~ -1, ~ -1, ~ 1, ~ 1)$, 
\end{tabularx}
}}

\vskip 1pt

{\hspace{-1.5cm}{
\begin{tabularx}{12cm}{X}
$(-1, ~\frac{103846370986858648996018705856006370440874894767103327090020707882131189518988898369586056510619}
{26494589591670393482362743865355713097178879313294723961579798826906175262124616047997572353148}$\\ 
\\ \vspace{-0.6cm}
$-\frac{24362602211847468548930474259939231149338256827219155205281311401412663732615050225593339451175}
{26494589591670393482362743865355713097178879313294723961579798826906175262124616047997572353148}$,~ $ -1, ~ 0, ~ -1, ~ -1, ~ 3, ~ 1, ~ 1)$, 
\end{tabularx}
}}

\vskip 1pt

{\hspace{-1cm}{
\begin{tabularx}{12cm}{X}
$(-1, ~ -1, ~ -1, ~ -1, ~ 0$, 
$\frac{3102512192376358142679819001580668439063486709837887008636819750620191848141353615683920698011}
{26361865615542855043385564990037983642532586413304692238940737587852090282522090736162403036796}$,\\ 
\\ \vspace{-0.6cm}
$\frac{49621219038709351944091310978495298846001686116771497469244655425083988716902827856640885375581}
{26361865615542855043385564990037983642532586413304692238940737587852090282522090736162403036796}$, ~ $ -1, ~1, ~ 1)$, 
\end{tabularx}
}}

\vskip 1pt

{\hspace{-1cm}{
\begin{tabularx}{12cm}{X}
$(-1, ~ -1, ~ -1, ~ -1, ~ 0$, 
$\frac{24711648058177108675870140414207889139439238180754794500520860494643214624518198751663039346627}
{26361865615542855043385564990037983642532586413304692238940737587852090282522090736162403036796}$,\\ 
\\ \vspace{-0.6cm}
$\frac{28012083172908601410900989565868078145625934645854589977360614681060965940525982720661766726965}
{26361865615542855043385564990037983642532586413304692238940737587852090282522090736162403036796}$,~ $ -1, ~ 0, ~ 1)$, 
\end{tabularx}
}}

\vskip 1pt

{\hspace{-1cm}{
\begin{tabularx}{12cm}{X}
$(-1, ~ -1, ~ -1, ~ -1, 0$,  
$\frac{33850383858990434912344303273412158256010059833790826610057805562903582596389175652598386119915}
{26361865615542855043385564990037983642532586413304692238940737587852090282522090736162403036796}$,\\ 
\\ \vspace{-0.6cm}
$\frac{18873347372095275174426826706663809029055112992818557867823669612800597968655005819726419953677}
{26361865615542855043385564990037983642532586413304692238940737587852090282522090736162403036796}$, ~ $ -1, ~ 1, ~ -1)$, 
\end{tabularx}
}}

\vskip 1pt

{\hspace{-1cm}{
\begin{tabularx}{12cm}{X}
$(-1, ~ -1, ~ -1, ~ -1, ~ 0$, 
$\frac{55459519724791185445534624686039378956385811304707734101941846306926605372766020788577504768531}
{26361865615542855043385564990037983642532586413304692238940737587852090282522090736162403036796}$\\ 
\\ \vspace{-0.6cm}
$-\frac{2735788493705475358763494705963411671320638478098349624060371131222424807721839316252698694939}
{26361865615542855043385564990037983642532586413304692238940737587852090282522090736162403036796}$,~ $ -1, ~ 0, ~ -1)$, 
\end{tabularx}
}}

\vskip 1pt

{\hspace{-1cm}{
\begin{tabularx}{12cm}{X}
$(-1, ~ -1, ~ -1, ~ -1, ~ 1$,  
$\frac{18978894299566317717088067776066397114618162390601452584064037808709694140758240615858478445167}
{26361865615542855043385564990037983642532586413304692238940737587852090282522090736162403036796}$,\\ 
\\ \vspace{-0.6cm}
$\frac{7382971315976537326297497213971586527914424022703239654876699779142396141763850120303924591629}
{26361865615542855043385564990037983642532586413304692238940737587852090282522090736162403036796}$, ~ $-1, ~ -1, ~ 1)$, 
\end{tabularx}
}}

\vskip 1pt

{\hspace{-1cm}{
\begin{tabularx}{12cm}{X}
$(-1, ~ -1, ~ -1, ~ -1, ~ 1$,  
$\frac{49726765966180394486752552047897886931564735514554392185485023620993084889006062652772943867071}
{26361865615542855043385564990037983642532586413304692238940737587852090282522090736162403036796}$,\\ 
\\ \vspace{-0.6cm}
$-\frac{23364900350637539443366987057859903289032149101249699946544286033140994606483971916610540830275}
{26361865615542855043385564990037983642532586413304692238940737587852090282522090736162403036796}$, ~ $ -1, ~ -1, ~ -1)$, 
\end{tabularx}
}}

\vskip 1pt

{\hspace{-1cm}{
\begin{tabularx}{12cm}{X}
$(-1, ~ 4, ~ -1, ~ -1, ~ 0$,  
$-\frac{6688458810009141025753858143779409510456031715576309136903886131445162415125537861347282349009}
{26361865615542855043385564990037983642532586413304692238940737587852090282522090736162403036796}$,\\ 
\\ \vspace{-0.6cm} 
$\frac{59412190041094851112524988123855376795521204542185693614785361307149342980169719333672088422601}
{26361865615542855043385564990037983642532586413304692238940737587852090282522090736162403036796}$, ~ $-1, 1, 1)$, 
\end{tabularx}
}}

\vskip 1pt

{\hspace{-1cm}{
\begin{tabularx}{12cm}{X}
$(-1, ~ 4, ~ -1, ~ -1, ~ 0$,  
$\frac{14920677055791609507436463268847811189919719755340598354980154612577860361251307274631836299607}
{26361865615542855043385564990037983642532586413304692238940737587852090282522090736162403036796}$\\ 
\\ \vspace{-0.6cm}
$\frac{37803054175294100579334666711228156095145453071268786122901320563126320203792874197692969773985}
{26361865615542855043385564990037983642532586413304692238940737587852090282522090736162403036796}$, ~ $-1, ~ 0, ~ 1)$, 
\end{tabularx}
}}

\vskip 1pt

{\hspace{-1cm}{
\begin{tabularx}{12cm}{X}
$(-1, ~ 4, ~ -1, ~ -1, ~ 0$, 
$\frac{24059412856604935743910626128052080306490541408376630464517099680838228333122284175567183072895}
{26361865615542855043385564990037983642532586413304692238940737587852090282522090736162403036796}$,\\ 
\\ \vspace{-0.6cm}
$\frac{28664318374480774342860503852023886978574631418232754013364375494865952231921897296757623000697}
{26361865615542855043385564990037983642532586413304692238940737587852090282522090736162403036796}$, ~ $-1, ~ 1, ~ -1)$,\\ 
\end{tabularx}
}}

\vskip 1pt
{\hspace{-1cm}{
\begin{tabularx}{12cm}{X}
$(-1, ~ 4, ~ -1, ~ -1, ~ 0$, 
$\frac{45668548722405686277100947540679301006866292879293537956401140424861251109499129311546301721511}
{26361865615542855043385564990037983642532586413304692238940737587852090282522090736162403036796}$,\\ 
\\ \vspace{-0.6cm}
$\frac{7055182508680023809670182439396666278198879947315846521480334750842929455545052160778504352081}
{26361865615542855043385564990037983642532586413304692238940737587852090282522090736162403036796}$, ~ $ -1, ~ 0, ~ -1)$, 
\end{tabularx}
}}

\vskip 1pt

{\hspace{-1cm}{
\begin{tabularx}{12cm}{X}
$(-1, ~ 4, ~ -1, ~ -1, ~ 1$,  
$\frac{9187923297180818548654390630706319165098643965187256438523331926644339877491349138827275398147}
{26361865615542855043385564990037983642532586413304692238940737587852090282522090736162403036796}$,\\ 
\\ \vspace{-0.6cm}
$\frac{17173942318362036494731174359331664477433942448117435800417405661207750405030741597335127638649}
{26361865615542855043385564990037983642532586413304692238940737587852090282522090736162403036796}$,~ $ -1, ~ -1, ~ 1)$, 
\end{tabularx}
}}

\vskip 1pt

{\hspace{-1cm}{
\begin{tabularx}{12cm}{X}
$(-1, ~ 4, ~ -1, ~ -1, ~ 1$, 
$\frac{39935794963794895318318874902537808982045217089140196039944317738927730625739171175741740820051}
{26361865615542855043385564990037983642532586413304692238940737587852090282522090736162403036796}$,\\ 
\\ \vspace{-0.6cm} 
$-\frac{13573929348252040274933309912499825339512630675835503801003580151075640343217080439579337783255}
{26361865615542855043385564990037983642532586413304692238940737587852090282522090736162403036796}$,~ $ -1, ~ -1, ~ -1)$, 
\end{tabularx}
}}

\vskip 1pt

{\hspace{-1cm}{
\begin{tabularx}{12cm}{X}
$(-\frac{26195047134744415782698212697367892027400682834912437463016463397578257179370909349948241063023}
{29609225983975155389271462639831245453881798006328279021067064863476984563956883782121934774784}$, ~ $ -1, ~ -1, ~ -1$,  \\ 
\\ \vspace{-0.6cm}
$-\frac{3414178849230739606573249942463353426481115171415841558050601465898727384585974432173693711761}
{29609225983975155389271462639831245453881798006328279021067064863476984563956883782121934774784}$,~ $1, ~ -1, ~ -1, ~ 1, ~ -1)$, 
\end{tabularx}
}}

\vskip 1pt

{\hspace{-1cm}{
\begin{tabularx}{12cm}{X}
$(-\frac{26195047134744415782698212697367892027400682834912437463016463397578257179370909349948241063023}
{29609225983975155389271462639831245453881798006328279021067064863476984563956883782121934774784}$, ~ $ -1, ~ -1, ~ -1$, \\ 
\\ \vspace{-0.6cm}
$-\frac{3414178849230739606573249942463353426481115171415841558050601465898727384585974432173693711761}
{29609225983975155389271462639831245453881798006328279021067064863476984563956883782121934774784}$, ~  $ 1, ~ -1$, \\ 
\\ \vspace{-0.6cm}
$\frac{33023404833205894995844712582294598880362913177744120579117666329375711948542858214295628486545}
{29609225983975155389271462639831245453881798006328279021067064863476984563956883782121934774784}$, ~ $ 1, ~ -1)$, 
\end{tabularx}
}}

\vskip 1pt

{\hspace{-1cm}{
\begin{tabularx}{12cm}{X}
$(-1, ~ -1$, 
$-\frac{21539430126287978048657430504105617153774538316156892440075206095782008733433118933039469537223}
{24536395391193293648676008436283697507274975628211884732471657650493104409471237752591331743744}$,\\ 
\\ \vspace{-0.6cm}
$ -1, ~ 1, ~ 2, ~ -1, ~ -1, ~ -1, ~ -1)$,
\end{tabularx}
}}

\vskip 1pt

{\hspace{-1cm}{
\begin{tabularx}{12cm}{X}
$(-1, ~ -1$,  
$-\frac{21122216541962554042102758493820344080793860456796043174421056184594377024885263320417638031983}
{24536395391193293648676008436283697507274975628211884732471657650493104409471237752591331743744}$,\\ 
\\ \vspace{-0.6cm}
$ -1, ~ 0, ~ 1, ~ -1, ~ -1, ~ 1, ~ -1)$,
\end{tabularx}
}}

\vskip 1pt

{\hspace{-1cm}{
\begin{tabularx}{12cm}{X}
$(-1, ~ -1$,  
$-\frac{21122216541962554042102758493820344080793860456796043174421056184594377024885263320417638031983}
{24536395391193293648676008436283697507274975628211884732471657650493104409471237752591331743744}$,\\ 
\\ \vspace{-0.6cm}
$ -1, ~ 0, ~ 1, ~ -1, ~ 1, ~ 1, ~ -1)$,
\end{tabularx}
}}

\vskip 1pt

{\hspace{-1cm}{
\begin{tabularx}{12cm}{X}
$(-1, ~ -1$,  
$-\frac{20559406117419291600070601443374888071110297468391335270675080253677555755818406397418193023943}
{24536395391193293648676008436283697507274975628211884732471657650493104409471237752591331743744}$,\\ 
\\ \vspace{-0.6cm}
$ -1, ~ 2, ~ 1, ~ -1, ~ -1, ~ -1, ~ -1)$,
\end{tabularx}
}}

\vskip 1pt

{\hspace{-1cm}{
\begin{tabularx}{12cm}{X}
$(-1, ~ -1$,  
$-\frac{11284861944552172578057845099108073886400579450196807847356973239687966709933002345360737249435}
{24536395391193293648676008436283697507274975628211884732471657650493104409471237752591331743744}$, \\ 
\\ \vspace{-0.6cm}
$ -1, ~ 0, ~ 0, ~ -1, ~ -1, ~ 1, ~ 1)$,
\end{tabularx}
}}

\vskip 1pt

{\hspace{-1cm}{
\begin{tabularx}{12cm}{X}
$(-1, ~ -1$,  
$-\frac{11284861944552172578057845099108073886400579450196807847356973239687966709933002345360737249435}
{24536395391193293648676008436283697507274975628211884732471657650493104409471237752591331743744}$,\\ 
\\ \vspace{-0.6cm}
$ -1, ~ 0, ~ 0, ~ -1, ~ 2, ~ 1, ~ 1)$,
\end{tabularx}
}}

\vskip 1pt

{\hspace{-1cm}{
\begin{tabularx}{12cm}{X}
$(-1, ~ -1$, 
$-\frac{4205446275243566766651530246728414079972770058649308300564461505431407404079782843339683799711}
{24536395391193293648676008436283697507274975628211884732471657650493104409471237752591331743744}$, \\ 
\\ \vspace{-0.6cm}
$ -1, ~ 3, ~ 0, ~ -1, ~ -1, ~ -1, ~ 0)$,
\end{tabularx}
}}

\vskip 1pt

{\hspace{-1cm}{
\begin{tabularx}{12cm}{X}
$(-1, ~ -1 $,
$-\frac{910318269355913964053938152209125536063027693005542117591291193863438934670986332681627401887}
{24536395391193293648676008436283697507274975628211884732471657650493104409471237752591331743744}$, \\ 
\\ \vspace{-0.6cm}
$ -1, ~ 0, ~ 3, ~ -1, ~ -1, ~ 0, ~ -1)$,
\end{tabularx}
}}

\vskip 1pt

{\hspace{-1cm}{
\begin{tabularx}{12cm}{X}
$(-1, ~ -1$,  
$\frac{9208441540326098721007053767725872663172034807796047161345779716501382014814703103874995884681}
{24536395391193293648676008436283697507274975628211884732471657650493104409471237752591331743744}$,\\ 
\\ \vspace{-0.6cm}
$ -1, ~ 1, ~ 2, ~ -1, ~ -1, ~ -1, ~ 1)$,
\end{tabularx}
}}

\vskip 1pt

{\hspace{-1cm}{
\begin{tabularx}{12cm}{X}
$(-1, ~ -1$,  
$\frac{11168489558063471618180711889187330828500516503327161500146031400710287970044128175117548911241}
{24536395391193293648676008436283697507274975628211884732471657650493104409471237752591331743744}$, \\ 
\\ \vspace{-0.6cm}
$ -1, ~ 3, ~ 0, ~ -1, ~ -1, ~ -1, ~ 1)$,
\end{tabularx}
}}

\vskip 1pt

{\hspace{-1cm}{
\begin{tabularx}{12cm}{X}
$(-1, ~ -1$,  
$\frac{20698817596444836569136383260418095164312723777911365374292749550159583841705858803297491246729}
{24536395391193293648676008436283697507274975628211884732471657650493104409471237752591331743744}$, \\ 
\\ \vspace{-0.6cm}
$ -1, ~ 0, ~ 3, ~ -1, ~ -1, ~ 1, ~ -1)$,
\end{tabularx}
}}

\vskip 1pt

{\hspace{-1cm}{
\begin{tabularx}{12cm}{X}
$(-1, ~ -1$,  
$\frac{22658865614182209466310041381879553329641205473442479713093001234368489796935283874540044273289}
{24536395391193293648676008436283697507274975628211884732471657650493104409471237752591331743744}$,\\ 
\\ \vspace{-0.6cm}
$ -1, ~ 2, ~ 1, ~ -1, ~ -1, ~ 1, ~ -1)$,
\end{tabularx}
}}

\vskip 1pt

{\hspace{-1cm}{
\begin{tabularx}{12cm}{X}
$(-1, ~ -1$,  
$\frac{51446689263058913338800867532249584981259296901864304975713735362442974589953680840211956668633}
{24536395391193293648676008436283697507274975628211884732471657650493104409471237752591331743744}$, \\ 
\\ \vspace{-0.6cm}
$ -1, ~ 0, ~ 3, ~ -1, ~ -1, ~ 1, ~ 1)$,
\end{tabularx}
}}

\vskip 1pt

{\hspace{-1cm}{
\begin{tabularx}{12cm}{X}
$(-1, ~ -1$,  
$\frac{54386761289664972684561354714441772229252019445160976483914112888756333522797818447075786208473}
{24536395391193293648676008436283697507274975628211884732471657650493104409471237752591331743744}$,\\ 
\\ \vspace{-0.6cm}
$-1, ~ 3, ~ 0, ~ -1, ~ -1, ~ 1, ~ 1)$,
\end{tabularx}
}}

\vskip 1pt

{\hspace{-1cm}{
\begin{tabularx}{12cm}{X}
$(-1, ~ 0$,  
$-\frac{1267228065897780484377965756925669407404625525800911902056194486913883573811692012312166677027}
{24536395391193293648676008436283697507274975628211884732471657650493104409471237752591331743744}$,\\ 
\\ \vspace{-0.6cm}
$ -1, ~ 4, ~ -1, ~ -1, ~ -1, ~ -1, ~ 0)$,
\end{tabularx}
}}

\vskip 1pt

{\hspace{-1cm}{
\begin{tabularx}{12cm}{X}
$(-1, ~ 0$,  
$\frac{14106707767409257900454276378990075501068661036175557898654298419227811800312219006145066033925}
{24536395391193293648676008436283697507274975628211884732471657650493104409471237752591331743744}$, \\ 
\\ \vspace{-0.6cm}
$ -1, ~ 4, ~ -1, ~ -1, ~ -1, ~ -1, ~ 1)$,
\end{tabularx}
}}

\vskip 1pt

{\hspace{-1cm}{
\begin{tabularx}{12cm}{X}
$(-1, ~ 0$,  
$\frac{57324979499010758966834919204244516901820163978009372882422379907273857353065909278103303331157}
{24536395391193293648676008436283697507274975628211884732471657650493104409471237752591331743744}$, \\ 
\\ \vspace{-0.6cm}
$ -1, ~ 4, ~ -1, ~ -1, ~ -1, ~ 1, ~ 1)$,
\end{tabularx}
}}

\vskip 1pt

{\hspace{-1cm}{
\begin{tabularx}{12cm}{X}
$(-1, ~ 1$,  
$\frac{28535302032873782030857170361485042674777494539139272510109535271403537457471465536595078518657}
{24536395391193293648676008436283697507274975628211884732471657650493104409471237752591331743744}$, \\ 
\\ \vspace{-0.6cm}
$ -1, ~ 4, ~ -1, ~ -1, ~ -1, ~ 1, ~ -1)$,
\end{tabularx}
}}

\vskip 1pt

{\hspace{-1cm}{
\begin{tabularx}{12cm}{X}
$(-1, -\frac{5521187510787516971870870347068480459912414277177670543342745837694983521558314717548661689685}
{22023265054199920193183578813167678159863825865109293166396089948535518390512798909314010468988}, ~ -1$\\ 
\\ \vspace{-0.6cm}
$\frac{71590982673387277551421606786571514939503891872505550042531015683301538693096711445490693096649}
{22023265054199920193183578813167678159863825865109293166396089948535518390512798909314010468988}$,\\
\\ \vspace{-0.9cm}
$\frac{82571872706012163800863444905602232179542889183259502122241613956447090040492880919707380186267}
{22023265054199920193183578813167678159863825865109293166396089948535518390512798909314010468988}$,\\ 
\\ \vspace{-1.2cm}
$-\frac{16502077543412403221312708466099197699951411587931622623053344110840534868954484191765348779303}
{22023265054199920193183578813167678159863825865109293166396089948535518390512798909314010468988}, ~ -1, ~ -1, ~ 1, ~ 1)$ 
\end{tabularx}
}}

\vspace{-0.5cm}

{\hspace{-1cm}{
\begin{tabularx}{12cm}{X}
$(-1, ~ -1, ~ -1, ~ -1$, 
$\frac{2122402386665917314214078798740121287598890383961392915595808495336557674585675611226638216643}
{21890541078072381754206399937849948705217532965119261443757028709481433410910273597478841152636}$, \\ 
\\ \vspace{-0.6cm}
 $ 0, ~ -1, ~ -1, ~ 1, ~ 0)$,
\end{tabularx}
}}

\vskip 1pt

{\hspace{-1cm}{
\begin{tabularx}{12cm}{X}
$(-1, ~ -1, ~ -1, ~ -1$,  
$\frac{2122402386665917314214078798740121287598890383961392915595808495336557674585675611226638216643}
{21890541078072381754206399937849948705217532965119261443757028709481433410910273597478841152636}$, \\ 
\\ \vspace{-0.6cm}
$ 0, ~  -1$,  
$\frac{41658679769478846194198721076959776122836175546277129971918248923626309147234871583731044088629}
{21890541078072381754206399937849948705217532965119261443757028709481433410910273597478841152636}$, ~ $1, ~ 0)$, 
\end{tabularx}
}}

\vskip 1pt

{\hspace{-1cm}{
\begin{tabularx}{12cm}{X}
$(-1, ~ -1, ~ -1, ~ -1$, 
$\frac{29966738284960379995762479488078817779877106763818738098653397077240907853215454864727642802923}
{21890541078072381754206399937849948705217532965119261443757028709481433410910273597478841152636}$, \\ 
\\ \vspace{-0.6cm}
$0, ~ -1, ~ -1, ~ -1, ~ 1)$,
\end{tabularx}
}}

\vskip 1pt

{\hspace{-1cm}{
\begin{tabularx}{12cm}{X}
$(-1, ~ -1, ~ -1, ~ -1$, 
$\frac{29966738284960379995762479488078817779877106763818738098653397077240907853215454864727642802923}
{21890541078072381754206399937849948705217532965119261443757028709481433410910273597478841152636}$,\\ 
\\ \vspace{-0.6cm}
$0, ~ -1$, 
$\frac{13814343871184383512650320387621079630557959166419784788860660341721958968605092330230039502349}
{21890541078072381754206399937849948705217532965119261443757028709481433410910273597478841152636}$,~ $-1, ~ 1)$, 
\end{tabularx}
}}

\vskip 1pt

{\hspace{-1cm}{
\begin{tabularx}{12cm}{X}
$(-1, ~ -1, ~ -1, ~ -1$,  
$\frac{39804092882370761459807392882791087974270387770417973425717480022147318168167715839784543585471}
{21890541078072381754206399937849948705217532965119261443757028709481433410910273597478841152636}$,\\ 
\\ \vspace{-0.6cm}
$1, ~ -1, ~ -1, ~ -1, ~ -1)$,
\end{tabularx}
}}

\vskip 1pt

{\hspace{-1cm}{
\begin{tabularx}{12cm}{X}
$(-1, ~ -1, ~ -1, ~ -1$,  
$\frac{39804092882370761459807392882791087974270387770417973425717480022147318168167715839784543585471}
{21890541078072381754206399937849948705217532965119261443757028709481433410910273597478841152636}$,\\ 
\\ \vspace{-0.6cm}
$1, ~ -1$,  ~ $-\frac{17913551804298379705600992944941139269052854805298711981960451312665884757257442242305702432835}
{21890541078072381754206399937849948705217532965119261443757028709481433410910273597478841152636}$, ~ $-1, ~ -1)$, 
\end{tabularx}
}}

\vskip 1pt

{\hspace{-1cm}{
\begin{tabularx}{12cm}{X}
$(-1, ~ -1, ~ -1, ~ -1$,  
$\frac{45340674118267418380594721623994562688350393325795207899363889983382603227339365883184875513875}
{21890541078072381754206399937849948705217532965119261443757028709481433410910273597478841152636}$,\\ 
\\ \vspace{-0.6cm}
$0, ~ -1, ~ -1, ~ -1, ~ 0)$,
\end{tabularx}
}}

\vskip 1pt

{\hspace{-1cm}{
\begin{tabularx}{12cm}{X}
$(-1, ~ -1, ~ -1, ~ -1$, 
$\frac{45340674118267418380594721623994562688350393325795207899363889983382603227339365883184875513875}
{21890541078072381754206399937849948705217532965119261443757028709481433410910273597478841152636}$, \\ 
\\ \vspace{-0.6cm}
$0, ~ -1$, ~
$-\frac{1559591962122654872181921748294665277915327395556685011849832564419736405518818688227193208603}
{21890541078072381754206399937849948705217532965119261443757028709481433410910273597478841152636}$, ~ $-1, ~ 0)$,\\ 
\end{tabularx}
}}

\vskip 1pt

{\hspace{-1cm}{
\begin{tabularx}{12cm}{X}
$(-1, ~ 0, ~ -1, ~ -1$,  
$\frac{5700789422085474401314672110871580411774992254255788160134077280158771881103947359220729535643}
{21890541078072381754206399937849948705217532965119261443757028709481433410910273597478841152636}$,\\ 
\\ \vspace{-0.6cm}
$-1, ~ -1, ~ -1, ~ 1, ~ 1)$,
\end{tabularx}
}}

\vskip 1pt

{\hspace{-1cm}{
\begin{tabularx}{12cm}{X}
$(-1, ~ 0, ~ -1, ~ -1$, 
$\frac{5700789422085474401314672110871580411774992254255788160134077280158771881103947359220729535643}
{21890541078072381754206399937849948705217532965119261443757028709481433410910273597478841152636}$,\\ 
\\ \vspace{-0.6cm}
$ -1, ~ -1$,  
$\frac{59970833812131670861304527702678265703877606641101996171137008848285528351626873433215793922265}
{21890541078072381754206399937849948705217532965119261443757028709481433410910273597478841152636}$,~  $1, ~ 1)$, 
\end{tabularx}
}}

\vskip 1pt

{\hspace{-1cm}{
\begin{tabularx}{12cm}{X}
$(-1, ~ 0, ~ -1, ~ -1$,  
$\frac{21074725255392512786146914246787325320248278816232257960844570186300467255227858377677962246595}
{21890541078072381754206399937849948705217532965119261443757028709481433410910273597478841152636}$, \\ 
\\ \vspace{-0.6cm}
$ -1, ~ -1, ~ -1, ~ 1, ~ 0)$,
\end{tabularx}
}}

\vskip 1pt

{\hspace{-1cm}{
\begin{tabularx}{12cm}{X}
$(-1, 0, -1, -1$, 
$\frac{21074725255392512786146914246787325320248278816232257960844570186300467255227858377677962246595}
{21890541078072381754206399937849948705217532965119261443757028709481433410910273597478841152636}$, \\
 
\\ \vspace{-0.6cm}
$ -1, ~ -1$,~   
$\frac{44596897978824632476472285566762520795404320079125526370426515942143832977502962414758561211313}
{21890541078072381754206399937849948705217532965119261443757028709481433410910273597478841152636}$, ~ $1, ~ 0)$, \\ 
\end{tabularx}
}}

\vskip 1pt

{\hspace{-1cm}{
\begin{tabularx}{12cm}{X}
$(-1, ~ 0, ~ -1, ~ -1$,  
$\frac{48919061153686975467695314936126021812526495196089603143902158768204817433857637631178966832875}
{21890541078072381754206399937849948705217532965119261443757028709481433410910273597478841152636}$, \\ 
\\ \vspace{-0.6cm}
$ -1, ~ -1, ~ -1, ~ -1, ~ 1)$,
\end{tabularx}
}}

\vskip 1pt

{\hspace{-1cm}{
\begin{tabularx}{12cm}{X}
$(-1, 0, -1, -1$,  
$\frac{48919061153686975467695314936126021812526495196089603143902158768204817433857637631178966832875}
{21890541078072381754206399937849948705217532965119261443757028709481433410910273597478841152636}$,\\ 
\\ \vspace{-0.6cm}
$ -1, ~ -1$, ~ 
$\frac{16752562080530169794923884877423824303126103699268181187368927360239482798873183161257556625033}
{21890541078072381754206399937849948705217532965119261443757028709481433410910273597478841152636}$, ~$ -1, ~ 1)$, \\ 
\end{tabularx}
}}

\vskip 1pt

{\hspace{-1cm}{
\begin{tabularx}{12cm}{X}
$(-1, 0, -1, -1$, 
$\frac{64292996986994013852527557072041766720999781758066072944612651674346512807981548649636199543827}
{21890541078072381754206399937849948705217532965119261443757028709481433410910273597478841152636}$, \\ 
\\ \vspace{-0.6cm}
$ -1, ~ -1, ~ -1, ~ -1, ~ 0)$,
\end{tabularx}
}}

\vskip 1pt

{\hspace{-1cm}{
\begin{tabularx}{12cm}{X}
$(-1, ~ 0, ~ -1, ~ -1$,  
$\frac{64292996986994013852527557072041766720999781758066072944612651674346512807981548649636199543827}
{21890541078072381754206399937849948705217532965119261443757028709481433410910273597478841152636}$, \\ 
\\ \vspace{-0.6cm} 
$ -1, ~ -1$,  $\frac{1378626247223131410091642741508079394652817137291711386658434454097787424749272142800323914081}
{21890541078072381754206399937849948705217532965119261443757028709481433410910273597478841152636}$, ~ $ -1, ~ 0)$,  
\end{tabularx}
}}

\vskip 1pt

{\hspace{-1cm}{
\begin{tabularx}{12cm}{X}
$(-1, ~ 0, ~ -1, ~ 3$,  
$\frac{82040976801502010044954729404331314360957717583299375231685369000230750122082779672366702920859}
{21890541078072381754206399937849948705217532965119261443757028709481433410910273597478841152636}$, \\ 
\\ \vspace{-0.6cm}
$ -1, ~ -1, ~ -1, ~ 1, ~ 1)$,
\end{tabularx}
}}

\vskip 1pt

{\hspace{-1cm}{
\begin{tabularx}{12cm}{X}
$(-1, ~ 0, ~ -1, ~ 3$,  
$\frac{82040976801502010044954729404331314360957717583299375231685369000230750122082779672366702920859}
{21890541078072381754206399937849948705217532965119261443757028709481433410910273597478841152636}$,\\ 
\\ \vspace{-0.6cm}
$ -1, ~ -1$,  
$-\frac{16369353567284864782335529590781468245305118687941590900414282871786449889351958879930179462951}
{21890541078072381754206399937849948705217532965119261443757028709481433410910273597478841152636}$, ~ $ 1, ~1)$, \\ 
\end{tabularx}
}}

{\hspace{-1cm}{
\begin{tabularx}{12cm}{X}
$(-1, ~ 1, ~ -1, ~ -1$,  
$\frac{34490466888222451337292420953631054638817661693125888532446921916029091776698391100728954348143}
{21890541078072381754206399937849948705217532965119261443757028709481433410910273597478841152636}$, \\ 
\\ \vspace{-0.6cm}
$ -1, ~ -1, ~ -1, ~ 1, ~ -1),$
\end{tabularx}
}}

{\hspace{-1cm}{
\begin{tabularx}{12cm}{X}
$(-1, ~ 1, ~ -1, ~ -1$,  
$\frac{34490466888222451337292420953631054638817661693125888532446921916029091776698391100728954348143}
{21890541078072381754206399937849948705217532965119261443757028709481433410910273597478841152636}$, \\ 
\\ \vspace{-0.6cm}
$ -1, ~ -1$, ~ $\frac{31181156345994693925326778859918791476834937202231895798824164212415208456032429691707569109765}
{21890541078072381754206399937849948705217532965119261443757028709481433410910273597478841152636}$, ~ $1, ~ -1)$, 
\end{tabularx}
}}

{\hspace{-1cm}{
\begin{tabularx}{12cm}{X}
$(-1, ~ 1, ~ -1, ~ -1$,  
$\frac{77708738619823952403673063778885496039569164634959703516215003404075137329452081372687191645375}
{21890541078072381754206399937849948705217532965119261443757028709481433410910273597478841152636}$, \\ 
\\ \vspace{-0.6cm}
$ -1, ~ -1, ~ -1, ~ -1, ~ -1)$,
\end{tabularx}
}}

{\hspace{-1cm}{
\begin{tabularx}{12cm}{X}
$(-1, ~ 1, ~ -1, ~ -1$,  
$\frac{77708738619823952403673063778885496039569164634959703516215003404075137329452081372687191645375}
{21890541078072381754206399937849948705217532965119261443757028709481433410910273597478841152636}$,\\ 
\\ \vspace{-0.6cm}
$ -1, ~ -1$, 
$-\frac{12037115385606807141053863965335649923916565739601919184943917275630837096721260580250668187467}
{21890541078072381754206399937849948705217532965119261443757028709481433410910273597478841152636}$,~ $ -1, ~ -1)$,\\ 
\end{tabularx}
}}

{\hspace{-1cm}{
\begin{tabularx}{12cm}{X}
$(-1, ~ 4, ~ -1, ~ -1$,  
$\frac{28615884286791151836232214666415007869105950637877370845122498386789879218738256214510232519931}
{21890541078072381754206399937849948705217532965119261443757028709481433410910273597478841152636}$,\\ 
\\ \vspace{-0.6cm}
$ -1, ~ -1, ~ -1, ~ 1, ~ -1)$,
\end{tabularx}
}}

{\hspace{-1cm}{
\begin{tabularx}{12cm}{X}
$(-1, ~ 4, ~ -1, ~ -1$,  
$\frac{28615884286791151836232214666415007869105950637877370845122498386789879218738256214510232519931}
{21890541078072381754206399937849948705217532965119261443757028709481433410910273597478841152636}$, \\ 
\\ \vspace{-0.6cm}
$ -1, ~ -1$,  
$\frac{37055738947425993426386985147134838246546648257480413486148587741654421013992564577926290937977}
{21890541078072381754206399937849948705217532965119261443757028709481433410910273597478841152636}$,~ $1, ~ -1)$, \\ 
\end{tabularx}
}}

{\hspace{-1cm}{
\begin{tabularx}{12cm}{X}
$(-1, ~ 4, ~ -1, ~ -1$,  
$\frac{41086284351778576132948373219837959452910880455758246227469594062552534023244124449554004395259}
{21890541078072381754206399937849948705217532965119261443757028709481433410910273597478841152636}$, \\ 
\\ \vspace{-0.6cm}
$-1, ~ -1, ~ -1, ~ -1, ~ 1)$,
\end{tabularx}
}}

{\hspace{-1cm}{
\begin{tabularx}{12cm}{X}
$(-1, ~ 4, ~ -1, ~ -1$,  
$\frac{41086284351778576132948373219837959452910880455758246227469594062552534023244124449554004395259}
{21890541078072381754206399937849948705217532965119261443757028709481433410910273597478841152636}$, \\ 
\\ \vspace{-0.6cm}
$ -1, ~ -1$,  
$\frac{24585338882438569129670826593711886662741718439599538103801492065891766209486696342882519062649}
{21890541078072381754206399937849948705217532965119261443757028709481433410910273597478841152636}$, ~ $ -1, ~ 1)$, \\ 
\end{tabularx}
}}

{\hspace{-1cm}{
\begin{tabularx}{12cm}{X}
$(-1, ~ 4, ~ -1, ~ -1,  
\frac{71834156018392652902612857491669449269857453579711185828890579874835924771491946486468469817163}
{21890541078072381754206399937849948705217532965119261443757028709481433410910273597478841152636}$,\\ 
\\ \vspace{-0.6cm}
$ -1, ~ -1, ~ -1, ~ -1, ~ -1)$,
\end{tabularx}
}}

{\hspace{-1cm}{
\begin{tabularx}{12cm}{X}
$(-1, ~ 4, ~ -1, ~ -1$,  
$\frac{71834156018392652902612857491669449269857453579711185828890579874835924771491946486468469817163}
{21890541078072381754206399937849948705217532965119261443757028709481433410910273597478841152636}$,\\ 
\\ \vspace{-0.6cm}
$ -1, ~ -1$, 
$-\frac{6162532784175507639993657678119603154204854684353401497619493746391624538761125694031946359255}
{21890541078072381754206399937849948705217532965119261443757028709481433410910273597478841152636}$,~ $ -1, ~ -1)$, 
\end{tabularx}
}}

{\hspace{-1cm}{
\begin{tabularx}{12cm}{X}
$(-1, ~ -1, ~ -1, ~ -1, ~ 0, ~ 0, ~ -1, ~ -1$, \\
\\ \vspace{-0.6cm}
$\frac{8357602419159629462572158075451597079501355292901830606769356333217885076838609728748524154307}
{21609135865800750533190321412627220700375751470916907491884040744023022776376845135979118648616}, ~ -1)$,
\end{tabularx}
}}

{\hspace{-1cm}{
\begin{tabularx}{12cm}{X}
$(-1, ~ -1, ~ -1, ~ -1, ~ 0, ~ 0, ~ -1, ~ 2$, \\ 
\\ \vspace{-0.6cm}
$\frac{8357602419159629462572158075451597079501355292901830606769356333217885076838609728748524154307}
{21609135865800750533190321412627220700375751470916907491884040744023022776376845135979118648616}, ~ -1)$,
\end{tabularx}%
}}

}}

\end{document}